\documentclass{amsart}

\pdfoutput=1

\usepackage{graphicx}
\usepackage{amssymb}
\usepackage{hyperref}
\usepackage{amsmath,amssymb}
\usepackage{dsfont}\let\mathbb\mathds
\usepackage[english]{babel}

\newtheorem{theorem}{Theorem}[section]
\newtheorem{lemma}[theorem]{Lemma}
\newtheorem{corollary}[theorem]{Corollary}
\newtheorem{proposition}[theorem]{Proposition}

\theoremstyle{remark}
\newtheorem{rem}[theorem]{Remark}
\theoremstyle{definition}
\newtheorem{definition}[theorem]{Definition}

\DeclareMathOperator{\im}{Im}

\newcommand{\Z}{\mathbb{Z}}

\def\del{\partial}              

\def\bC{\mathbb C}          
\def\bR{\mathbb R}          
\def\bQ{\mathbb Q}          
\def\bN{\mathbb N}          
\def\bZ{\mathbb Z}          

\def\bcp{\mathbb C \mathbb P}

\def\mC{\mathcal{C}}
\def\mS{\mathcal{S}}            
\def\mL{\mathcal{L}}            
\def\mS{\mathcal{S}}            

\def\mD{\mathcal{D}}

\def\bfD{\mathbf{D}}  \def\ev{\mbox{ev}}
 \newcommand*{\quot}[2]%
{\ensuremath{%
   \raisebox{.35ex}{\ensuremath{#1}}\big/\raisebox{-.35ex}{\ensuremath{#2}}}}

\def\Lie{\mbox{Lie }}
\def\bfH{\mbox{{\bf H}}}
\def\kt{\mathfrak{t}}
\def\ra{\rightarrow}
\def\vol{\mbox{dv}}

\begin{document}

\title{Toric Generalized K\"ahler--Ricci Solitons with Hamiltonian $2$--form}
\author{Eveline Legendre and Christina W. T{\o}nnesen-Friedman}
\address{Eveline Legendre, I.M.T., Universit\'e Paul Sabatier, 31062 Toulouse cedex 09, France}
\email{eveline.legendre@math.univ-toulouse.fr}
\address{Christina W. T\o nnesen-Friedman, Department of Mathematics, Union
College, Schenectady, New York 12308, USA } \email{tonnesec@union.edu}
\thanks{The work of the first author was supported by NSERC grant
BP387479. The work of the second author was supported by a grant from
the Skidmore-Union SUN Network which is founded by
the National Science Foundation and ADVANCE Grant 0820032.}


\maketitle

\begin{abstract}
    We show that the generalized K\"ahler--Ricci soliton equation on $4$--dimensional toric K\"ahler orbifolds reduces to ODEs assuming there is a Hamiltonian $2$--form. This leads to an explicit resolution of this equation on labelled triangles and convex labelled quadrilaterals. In particular, we give the explicit expression of the K\"ahler-Ricci solitons of weighted projective planes as well as new examples.
\end{abstract}

\section{The K\"ahler-Ricci soliton problem}
A K\"ahler-Ricci soliton $(M,g,\omega,J)$ is a K\"ahler compact orbifold with a positive number $\lambda$ and a holomorphic vector field $X$ such that \begin{equation}\label{KRSequation}
  \rho -\lambda\omega =\mL_X\omega
\end{equation}
where $\rho$ is the Ricci form. The equation~(\ref{KRSequation}) determines that $$\lambda = \frac{1}{2n}\overline{Scal}$$ where $2n$ is the dimension of $M$ and $\overline{Scal}=\int_{M} Scal\, \omega^{n}/\int_M\omega^n$. It also implies that $(M,J)$ is Fano in the sense that $c_1(M)>0$ and $(M,\omega)$ is a monotone symplectic orbifold in the sense that $\lambda[\omega]=2\pi c_1(M)$ in $H_{\mathrm{dR}}^2(M,\bR)$. The case $X=0$ corresponds to K\"ahler--Einstein metrics. In the toric case, there is a vector $a\in \kt=\mbox{Lie }T$, uniquely determined by the data $(M,[\omega],T)$, such that if there exists a compatible K\"ahler--Ricci soliton with respect to the vector field $X$ then $X=X_a- iJX_a$, see~\cite{Z} and Lemma~\ref{remEQUAkrVECTOR}. Let us call this vector $a$ the {\it K\"ahler--Ricci soliton vector}.\\

The existence problem of K\"ahler--Ricci soliton on compact complex manifold
appeared in the study of the K\"ahler--Ricci flow and attracts interest also as
an obstruction to the existence of
K\"ahler--Einstein metric~\cite{Do:withlargesymmetry,TZ,WZ:krsontoric,Z}.
There exists a unique K\"ahler--Ricci soliton on any toric
compact Fano complex orbifold~\cite{SZ,TZ,WZ:krsontoric}. The proofs of Wang--Zhu and Shi--Zhu, and the
alternative proof given by Donaldson in~\cite{Do:withlargesymmetry}, use
a continuity method. In particular, it does not provide an explicit expression
of the K\"ahler--Ricci soliton.\\

Given a K\"ahler--Ricci soliton vector $a\in\kt$, taking the trace of equation~(\ref{KRSequation})
leads to the equation:
\begin{equation}\label{KRStraceequation}
  Scal -\overline{Scal} = 2\Delta^g\langle\mu,a\rangle.
\end{equation}
where $\mu : M\ra\kt^*$ is a moment map of the compact symplectic toric
orbifold $(M,\omega, T)$. If $(M,\omega,\mu,T,g,a)$ satisfies
\eqref{KRStraceequation} we say that it is a {\it generalized
K\"ahler--Ricci soliton}.  While \eqref{KRStraceequation} is equivalent
to \eqref{KRSequation} in the monotone (compact) case, it may also be studied
in its own right. D. Guan first introduced these metrics under the name
{\it generalized quasi-Einstein metrics} \cite{guan1,guan2}. As far
as we know, the problem of existence and uniqueness of generalized K\"ahler--Ricci solitons is wide open.

At this time, there are only a few types of non-K\"ahler--Einstein {\it explicit}
K\"ahler--Ricci solitons
(or generalized K\"ahler--Ricci
solitons) known. In \cite{DaWa}, Dancer and Wang unify and
generalize the construction of
cohomogeneity one Ricci solitons.
In~\cite{guan2}, Guan makes a parallel between extremal K\"ahler metrics and
generalized K\"ahler-Ricci solitons and uses Calabi's ansatz
(originally for extremal metrics) to construct explicit examples of
generalized K\"ahler-Ricci soliton on Hirzebruch surfaces and certain other
$\bcp^l$--bundles. Using a continuity argument, the existence results following
from this work has
been slightly extended in \cite{MT-F}. The Calabi extremal K\"ahler metric is
a special case of metrics with Hamiltonian $2$--form. Introduced and
classified in~\cite{wsd,H2FI,H2FII}, these metrics includes all
explicit examples of extremal K\"ahler metrics known so far.
Following Guan's idea, it seems then natural to look for generalized
K\"ahler-Ricci solitons admitting a Hamiltonian $2$--form. In this paper,
we focus on the $4$--dimensional toric case.

Our main theorem is as follows.

\begin{theorem}\label{theoWPP}The K\"ahler--Ricci soliton of a weighted projective plane is the
    Fubini-Study metric when the three weights coincide, it is a Calabi metric
    when only two weights coincide and an orthotoric metric if the weights are
    all distinct. In particular, a K\"ahler--Ricci soliton of a weighted projective plane admits a Hamiltonian 2-form and may be explicitly expressed.
\end{theorem}

While the existence of K\"ahler--Ricci solitons on the weighted projective plane follows
from \cite{SZ} the explicit expression for the metrics has until now remained unknown.

Via the orbifold Boothby--Wang construction the K\"ahler--Ricci solitons constructed on the weighted projective plane ``lift''  to Sasaki--Ricci solitons on the weighted 5-sphere as defined in
Example 7.1.12 of \cite{BG}. The existence of these metrics was proved by Futaki, Ono, and Wang in \cite{FuOnWa}. In the toric case, the Boothby--Wang construction applies more generally to non regular cases and we get the following result.

\begin{theorem}\label{theoS5}
Any toric Sasaki--Ricci soliton on $S^5$ admits a transversal Hamiltonian $2$--form whose type only depends on the Reeb vector field.
\end{theorem}

More recently He and Sun proved that the K\"ahler--Ricci solitons on a weighted projective space must have positive bisectional curvature (See Theorem 1.4 in  \cite{HeSun}). He and Sun mentions the fact that they were not able to find a general way of producing an explicit orbifold K\"ahler metric with positive bisectional curvature even on weighted projective spaces. Thus our examples may be viewed as a resolution to this problem...at least in complex dimension two.

In fact, this remark might motivate one to embark on a general construction of K\"ahler--Ricci solitons on weighted projective spaces in all dimensions. For instance, in the case where the weights are all distinct one could, in theory, use Theorem 3 in \cite{H2FII} to generalize the derivations in section \ref{sectionOrtho} and the calculations in section \ref{wpp}. Writing a generalization of \eqref{KRStraceequationorthotoric} is straightforward and we also arrive at a collection of ODE's generalizing \eqref{KRStraceequationorthotoric2}. However, we expect the requirement of an argument far more delicate than the one occurring in section \ref{wpp} before one could possibly arrive at an explicit K\"ahler--Ricci soliton.
%
When the weights are not all distinct we have to consider hamiltonian $2$--forms of orders less than the complex dimension of the weighted projective space. With higher dimension there are more special cases to consider.

\subsection{Technical Introduction}

The toric assumption allows one to translate the problem of finding generalized
K\"ahler-Ricci toric solitons into a PDE--problem defined on the moment polytope
(the image of the moment map). Due to the Delzant--Lerman--Tolman
correspondence~\cite{delzant:corres,LT:orbiToric}, the compact symplectic toric
orbifold $(M,\omega, T)$ corresponds to a {\it rational labelled polytope}
$(\Delta,u)$. Specifically, $\Delta=\im\mu$ is a simple convex compact polytope
in $\kt^*$ with $d$ codimension $1$ faces, the facets,
$F_1,\dots,F_d$, while
$u=\{u_1,\dots,u_d\}\subset \kt$ is a set of vectors, inward to
$\Delta$,
such that $u_k$ is normal to $F_k$. Then $(\Delta,u)$ being
rational means that $u$ lies in a lattice of $\kt$.
Following the work of Abreu and Guillemin~\cite{abreu,guillMET},
and as we recall in Section~\ref{sectBG}, compatible toric K\"ahler metrics
and their curvature amount to special functions on $\Delta$,
so equation \eqref{KRStraceequation} reads as a PDE-equation on
$\Delta$, precisely defined with the data $(\Delta,u)$.

In this paper, we show that the variables separate in
equation~\eqref{KRStraceequation} assuming that the dimension is $4$ and
that there is a Hamiltonian $2$--form coming with the K\"ahler structure
$(g,\omega,J)$. In the toric setting, there exist compatible K\"ahler metrics
admitting such a form if and only if the moment polytope is a triangle or a
quadrilateral~\cite{TGQ}. Excluding the Fubini--Study metric on $\bcp^2$
for a moment, we may say that
K\"ahler structures with
a Hamiltonian $2$--form are characterized by their local expression: there are
four functions $x,y\in C^{\infty}(M)$ and $A,B:\bR\ra \bR$ so the K\"ahler
metric can be explicitly written in terms of $x,y,A(x),B(y)$. The exact
expression depends on the order of the Hamiltonian $2$--form~\cite{H2FI}.
In dimension $4$, there are three orders $2$, $1$ or $0$ corresponding
respectively to the type of the metric: orthotoric, Calabi or a product metric.
Assuming the moment polytope is a quadrilateral, the order prescribes exclusively
the type of the moment polytope (in order): generic quadrilateral
(without parallel edges), trapezoid (with only one pair of parallel edges) or
parallelogram. The Fubini--Study metric on $\bcp^2$
(whose moment polytope is a triangle)
admits several different Hamiltonian
$2$--forms  \cite{H2FI} and in particular has a Hamiltonian $2$-form of each
order $2$, $1$, and $0$.

\begin{definition}
Let $\Delta$ be a quadrilateral with vertices $s_{1},\ldots,s_{4}$, such that $s_{1}$ is not consecutive to $s_{3}$. Then the affine
function $f$ on $\Delta$ is {\it equipoised} on $\Delta$ if $\sum_{i=1}^{4} (-1)^{i}f(s_{i}) =0.$
\end{definition}
For example, on a parallelogram every affine linear function is
equipoised.

\begin{theorem} \label{theoPrincipal} Let $(M,\omega,T)$ be a
symplectic toric $4$--orbifold with labelled polytope $(\Delta,u)$.
There is at most one generalized toric K\"ahler--Ricci soliton $(g,a)$, such
that $g$ admits a Hamiltonian $2$--form. In that case, $a$ is implicitly
determined by the data $(\Delta,u)$ and is equipoised if $\Delta$ is a
quadrilateral. Moreover, either $g$ is the Fubini--Study metric on $\bcp^2$ or, using an appropriate
identification $\kt\simeq \bR^2$, $g$ is explicitly given in terms of $a$
and two functions of one variable $A$ and $B$ given
\begin{itemize}
  \item[-] by~\eqref{fctA} and~\eqref{fctB} if $g$ is orthotoric,
  \item[-] by~\eqref{solAcalabi} and~\eqref{fctBcalabi} if $g$ is Calabi,
  \item[-] by~\eqref{solnPROD2A} and~\eqref{solnPROD2B} if $g$ is a product metric.
\end{itemize}
Any monotone compact symplectic toric orbifold, whose moment polytope is a
quadrilateral and K\"ahler--Ricci soliton vector $a$ is equipoised, admits a compatible
K\"ahler--Ricci soliton $(g,a)$ with a Hamiltonian $2$--form.
\end{theorem}

Our result provides an explicit expression for any $4$--dimensional
toric generalized K\"ahler--Ricci soliton admitting a Hamiltonian
$2$--form and which, whenever they are smooth, correspond to Guan's examples on the Hirzebruch
surfaces or the constant scalar curvature (CSC) K\"ahler metric on
$\bcp^1 \times \bcp^1$ or $\bcp^2$.

In the cases of parallelograms and trapezoids we obtain the following two existence results.

\begin{proposition}
\begin{enumerate}
\item Let $\Delta$ be a parallelogram. There is a $2$-rational by $1$-real
parameter family of generalized K\"ahler--Ricci
solitons compact toric $4$--orbifolds  $(M, \omega, T, g, a)$ having as common
moment polytope $\Delta$.
This family contains a $2$--rational
parameter family of K\"ahler--Ricci soliton toric $4$--orbifolds.

\item Let $\Delta$ be a trapezoid.
There is a
$2$--rational parameters family of generalized K\"ahler--Ricci
solitons compact toric $4$--orbifolds $(M, \omega, T, g, a)$ having $\Delta$
as common moment polytope.
This family contains a $1$--rational parameter subfamily of
K\"ahler--Ricci soliton toric $4$--orbifolds if and only if the
vertices of $\Delta$ lie in a lattice $\Lambda^* \subset \kt^*$.
\end{enumerate}

In both cases $(g, \omega, J)$ admits a
Hamiltonian $2$--form and $a$ is equipoised on $\Delta$.
\end{proposition}

In the parallelogram case, our classification is exhaustive (in particular,
it includes the second example of~\cite{SZ}) but in the trapezoid case
there are K\"ahler--Ricci solitons which do not admit a Hamiltonian $2$--form.
Shi and Zhu proved recently that each monotone toric orbifold admits a compatible K\"ahler--Ricci soliton~\cite{SZ}. However, it comes out of our study that a $4$--dimensional toric orbifold whose moment polytope is a trapezoid admits a generalized K\"ahler--Ricci soliton with a Hamiltonian $2$--form if and only if it admits a compatible extremal K\"ahler metrics with a Hamiltonian $2$-form, this corresponds to a non trivial condition on the Futaki invariant (restricted to the real torus) which is not always satisfied in the monotone case, see Proposition~\ref{ExistSOLcalabi}.

In the quadrilateral case, we will point out, in Section~\ref{rationality}, evidence of existence of orthotoric generalized K\"ahler--Ricci solitons. However, we have a non-existence conclusion of orthotoric K\"ahler--Ricci solitons on generic quadrilaterals.

\begin{proposition}\label{propostionNOexistOrtho}
Let $(\Delta,u)$ be a rational generic labelled quadrilateral whose vertices
lie in the lattice generated by $u$. Then a toric generalized K\"ahler--Ricci
soliton $(g,a)$ on the corresponding symplectic toric orbifold $(M,\omega,T)$
admits a Hamiltonian 2-form (thus given by the construction of
Theorem~\ref{theoPrincipal}) if and only if $a=0$, i.e. $g$ is a CSC K\"ahler metric. In particular, in the case when $(\Delta,u)$ is generic and monotone, the Shi--Zhu K\"ahler--Ricci soliton is given by \eqref{fctA} and~\eqref{fctB} if and only if it is K\"ahler--Einstein.
\end{proposition}

As we will recall in Section~\ref{sectBG}, there is no
need to assume any rational condition on the labelled polytope in order to
define equation \eqref{KRStraceequation}. It can also be geometrically
interesting to work in this generality, by thinking, for example,
of applications in Sasaki toric geometry. It makes sense to talk about a
generalized K\"ahler--Ricci soliton $(g,a)$ of a labelled polytope $(\Delta,u)$
and, in the non-rational case, one can think at $g$ as a metric defined on
$\mathring{\Delta}\times \bR^n$ satisfying specific boundary condition.
From the discussions in this paper the following proposition easily follows.

\begin{proposition}\label{propCONE}
  Let $\Delta$ be a convex quadrilateral. For
any equipoised linear function $a\in\kt$ there is a $2$--dimensional cone of
inward normals $u$ for which $(\Delta,u)$ admits a generalized K\"ahler--Ricci
soliton $(g,a)$ with a Hamiltonian $2$--form.
This cone contains a codimension $1$ subcone of inward normal $u$ for
which $(\Delta,u)$ is monotone.
\end{proposition}

Section~\ref{sectBG} contains a quick review explaining the translation of the geometric problem into a PDE on labelled polytopes. The reduction to ODEs and their resolution are done in Section \ref{sectRESOL}. In Section \ref{sectMONOT}, we deal with the monotone case which completes the proof of Theorem~\ref{theoPrincipal}. In Section~\ref{rationality} we consider the rational condition (needed to define orbifolds) on our set of solutions. Section~\ref{sectWPP} contains the argument for the weighted projective planes and thus the proof of Theorem~\ref{theoWPP}. Sections \ref{sectRESOL},~\ref{sectMONOT} and~\ref{rationality} are divided into three parts each, corresponding respectively to the type of quadrilaterals or, equivalently, the type of metrics. \\

We would like to express our sincere gratitude to Vestislav Apostolov for his
many helpful comments and suggestions during this project.

\section{Background on K\"ahler toric geometry}\label{sectBG}

A labelled polytope $(\Delta,u)$, with $\Delta\subset \kt^*$, is completely determined by its {\it defining functions}: the affine-linear functions $L_1,\dots, L_d$ on $\kt^*$ such that $\Delta=\{p\in\kt^*\,|\,L_k(p)\geq0\}$ and $dL_k=u_k$. For instance, two labelled polytopes are {\it equivalent} if there is an invertible affine map inducing a bijection on their set of defining functions.

Let $(M,\omega,T)$ be a symplectic toric orbifold with moment map $\mu : M\ra \kt^*$. It corresponds, via the Delzant-Lerman-Tolman
correspondence~\cite{delzant:corres,LT:orbiToric}, to a rational labelled polytope $(\Delta, u, \Lambda)$ where
$\Delta=\im \mu \subset \kt^*=(\mbox{Lie }T)^*$ and $T=\kt/\Lambda$. An equivariant symplectomorphism between symplectic toric orbifolds amounts the equivalence of corresponding labelled polytopes.

Abreu~\cite{abreu} showed that $T$--invariant $\omega$-compatible K\"ahler metrics correspond to {\it symplectic potentials} modulo affine-linear functions: A symplectic potential is a continuous function $\phi\in C^0(\Delta)$ whose restriction on $\mathring{\Delta}$ or any non-empty face's interior of $\Delta$ is smooth and strictly convex and  $\phi-\phi_o$ is the restriction of a smooth function defined on an open set containing $\Delta$ where $\phi_o = \sum_{k=1}^d L_k \log L_k$ is the Guillemin potential. Denote $\mS(\Delta,u)$ the set of symplectic potentials.

Recall that $\mathring{M}=\mu^{-1}(\mathring{\Delta})$ is a dense open subset of $M$ and is the set of points where the torus acts freely. Given a basis $(e_1,\dots, e_n)$ of $\kt$, we set $\mu_{i} =\langle\mu,e_i\rangle$ for $i=1,\dots,n$. The  {\it action-angle coordinates} $(\mu_1,\dots,\mu_n,t_1,\dots,t_n)$ are local coordinates on $\mathring{M}$ coming from an equivariant identification between the universal cover of $\mathring{M}$ and $\mathring{\Delta}\times \kt \simeq \mathring{\Delta}\times \bR^n$, the $1$--forms $dt_1$,$\dots$, $dt_n$ are globally defined on $\mathring{M}$, see for e.g.~\cite{abreu}. In these coordinates, on $\mathring{M}$ we have $\omega =\sum_{i=1}^n d\mu_i\wedge dt_i$ and for any $\phi\in \mS(\Delta,u)$ one can define the $T$--invariant $\omega$--compatible K\"ahler metric $g_{\phi}$ on $\mathring{M}$ as
\begin{align}\label{ActionAnglemetric}
g_{\phi} = \sum_{i,j} G_{ij}d\mu_i\otimes d\mu_j + H_{ij}dt_i\otimes dt_j,
\end{align}
where $(G_{ij})=\mbox{Hess }\phi$ and $(H_{ij})=(G_{ij})^{-1}$ are smooth on $\mathring{\Delta}$. The boundary behavior of $\phi$ ensures that $g_{\phi}$ is the restriction of a smooth $T$--invariant K\"ahler metric on $M$. Moreover, one can show that, up to an equivariant symplectomorphism, $g_{\phi}$ does not depend on the choice of the action-angle coordinates, see~\cite{abreu,H2FII}.\\

Apostolov et al. gave an alternative description of $\mS(\Delta,u)$ as follows.

\begin{proposition}\label{defSympPot}\cite{H2FII}The set of symplectic potentials $\mS(\Delta,u)$ is the space of smooth strictly convex functions $\phi$ defined on the interior of the polytope $\mathring{\Delta}$ for which $\bfH=(\mbox{Hess }\phi)^{-1}$ is the restriction to $\mathring{\Delta}$ of a smooth $S^2\mathfrak{t}^*$--valued function on $\Delta$, still denoted by $\bfH$, which satisfies the boundary condition: For every $y$ in the interior of the facet $F_i\subset \Delta$,
\begin{equation}\label{condCOMPACTIFonH}
  \bfH_y (u_i, \cdot\,) =0\;\;\;\mbox{ and }\;\;\; d\bfH_y (u_i, u_i) =2u_i
\end{equation}
and the positivity conditions: The restriction of $\bfH$ to the interior of any face $F \subset \Delta$ is a positive definite $S^2(\mathfrak{t}/\mathfrak{t}_F)^*$--valued function where $\mathfrak{t}/\mathfrak{t}_F$ is identified to $T^*F$.
\end{proposition}

The Abreu formula~\cite{abreu} is $Scal_g=\mu^*S(\bfH_{\phi})$ where
\begin{equation}\label{abreuFormula}
 S(\bfH_{\phi}) = -\sum_{i,j=1}^n \frac{\del^2 H_{ij}}{\del\mu_i\del\mu_j}.
\end{equation}
 In~\cite{don:scalar}, Donaldson pointed out that, for any euclidian volume form $dv$, the $L^2(\Delta,dv)$--projection of $S(\bfH_{\phi})$ on the space of affine-linear functions on $\Delta$ does not depend on the choice of $\phi$ in $S(\Delta,u)$. The resulting projection $\zeta_{(\Delta,u)}\in \mbox{Aff}(\Delta,\bR)$ will be called the {\it extremal affine function}. Moreover, we have $$\overline{Scal}=\frac{\int_{\Delta}S(\bfH_{\phi})dv}{\int_{\Delta}dv}=\frac{\int_{\Delta}\zeta_{(\Delta,u)}dv}{\int_{\Delta}dv}= \frac{2\int_{\del\Delta}d\ell_u}{\int_{\Delta}dv}$$
where $d\ell_u$ is defined by the property $u_k\wedge d\ell_u=-dv$.

In view of equation (\ref{KRStraceequation}), we compute the de Rham Laplacian in this setting:
\begin{lemma} Given a $T$--invariant $\omega$-compatible K\"ahler metric $g_{\phi}$ with $\bfH= (\mbox{Hess } \phi)^{-1}$ and $a\in\kt$, we have \begin{equation}\label{ToricLaplacian}
  \Delta^{g_{\phi}} \langle\mu,a\rangle = - \mathrm{div } \,\bfH (a,\cdot).
\end{equation}
\end{lemma}

The right hand side of (\ref{ToricLaplacian}) is a globally defined on $M$ since $\bfH(a,\cdot)$ is naturally identified with a vector field on $\Delta$ as a smooth $\kt^*$--valued function on $\Delta$.
\begin{proof} On $\mathring{M}$, writing $g_{\phi}=\sum_{r,s=1}^{2n} g_{rs}dx_r\otimes dx_s$ with $(g^{rs})=(g_{rs})^{-1}$ we have the well-known formula $\Delta^{g_{\phi}}= \frac{-1}{\sqrt{\det g_{\phi}}} \frac{\partial }{\partial x_r}\left(g_{\phi}^{rs}\sqrt{\det
g}\frac{\partial }{\partial x_s}\right).$ Using action-angle coordinates $(\mu_1,\dots,\mu_n,t_1,\dots,t_n)$ and~(\ref{ActionAnglemetric}), we obtain the local formula $$\Delta^{g_{\phi}} = -\sum_{i,j=1}^n \frac{\partial }{\partial \mu_i}\left(H_{ij}\frac{\partial }{\partial \mu_j}\right) + \frac{\partial }{\partial t_i}\left(G_{ij}\frac{\partial }{\partial t_j}\right).$$ Hence, since $\langle\mu,a\rangle= \sum_{i=1}^n a_i \mu_i$ is $T$--invariant as a function on $\mathring{M}$ and is affine-linear as a function on $\Delta$ we get
 $\Delta^{g_{\phi}} \langle\mu,a\rangle = -\sum_{i,j=1}^n a_j \frac{\partial H_{ij}}{\partial \mu_i}= - \mathrm{div } \,\bfH (a,\cdot).$ \end{proof}

Therefore, the problem of finding explicit generalized toric K\"ahler--Ricci soliton can be read as: {\it Given a labelled polytope $(\Delta,u)$ in $\kt^*$ and a vector $a\in\kt$ does there exist a symplectic potential $\phi\in \mS(\Delta,u)$ such that $$ S(\bfH_{\phi})+2\mathrm{div } \,\bfH (a,\cdot)=\overline{Scal}$$ and, in that case, what is $\phi$ explicitly ?}\\

There is another way to formulate the problem of finding K\"ahler--Ricci soliton in terms of labelled polytope which is more commonly used ~\cite{Do:withlargesymmetry,TZ,WZ:krsontoric,Z}. Let us recall it briefly. Given symplectic potential $\phi\in\mS(\Delta,u)$, the Ricci potential associated to $\phi$ is $F(\mu)= \frac{1}{2} \mathrm{log }\,\mathrm{det }\,(\mathrm{Hess }\,\phi)_{\mu}$, that is $F\in C^{\infty}(\Delta)$ and $\rho=dd^c F$, \cite{abreu}. Moreover, the Legendre transform of $\phi$ (seen as a function on $\mathring{\Delta}$, via the change of variable $\mu\mapsto (\mathrm{grad}\, \phi)_{\mu}\in\kt$) $$\psi(\mu)=\langle \mu, d_{\mu}\phi \rangle - \phi(\mu)$$ is a K\"ahler potential meaning $\omega=dd^c \psi$. It is straightforward to see that: {\it $(M,\omega)$ is monotone with constant $\lambda>0$ if and only if for any symplectic potential $\phi\in\mS(\Delta,u)$, $F-\lambda \psi$ is a smooth function on $\Delta$} (by their very definition, if there is a symplectic potential satisfying this statement then any symplectic potential does).
 This observation leads to the next lemma for which we need the definition:

\begin{definition}  We say that $(\Delta,u)$ is {\it monotone} if there
    exists $p\in\Delta$ such that $L_1(p)=L_2(p)=\dots=L_d(p)$. In that
    case, we call $p$ the {\it preferred point} of $(\Delta,u)$. \end{definition}

\begin{lemma}\label{lemMONOTONE}A compact symplectic toric orbifold is monotone if and only its associated labelled polytope is.
\end{lemma}
\begin{proof}

The algebraic counterpart of lemma~\ref{lemMONOTONE} (where the monotone condition is replaced by the Fano condition) is well-known, see for e.g. \cite{Do:withlargesymmetry, WZ:krsontoric}. To prove it~\footnote{Vestislav Apostolov gave us the idea of this proof.} in the symplectic case (orbifold or not) first notice that for any $p\in\kt^*$, $$\psi(\mu)= \langle \mu-p, d_{\mu}\phi \rangle - \phi(\mu)$$ is also a potential for $\omega$. The Guillemin potential $\phi =\frac{1}{2}\sum_{k=1}^d L_k(\mu)\mathrm{log }\,L_k(\mu)$ gives the Ricci potential $F(\mu) = \frac{-1}{2} \sum_{k=1}^d\mathrm{log }\, h(\mu)L_k(\mu)$ where $h$ is a strictly positive smooth function on $\Delta$, see~\cite{abreu}. Hence, one can compute that $$F-\lambda\psi +\frac{1}{2} \sum_{k=1}^d(1- \lambda L_k(p))\mathrm{ log }\, L_k$$ is a smooth function on $\Delta$ for any $p\in\kt^*$. In particular, $(M,\omega)$ is monotone if and only if $(\Delta,u)$ is.
\end{proof}

A symplectic potential $\phi\in\mS(\Delta,u)$ corresponds to a K\"ahler--Ricci soliton with respect to $\lambda>0$ and the K\"ahler--Ricci soliton vector $a\in\kt$ if and only if

\begin{equation}\label{eq:KRsolitonPOT}
  F -\lambda\psi =-\langle\mu,a\rangle.
\end{equation}

A vector $a\in\kt$ satisfying \eqref{eq:KRsolitonPOT} is uniquely (but implicitly) determined by the data $(\Delta,u)$ as claimed in the following lemma.

\begin{lemma}\label{remEQUAkrVECTOR} \cite{Do:withlargesymmetry,Z} Given a monotone labelled polytope $(\Delta,u)$ whose preferred point is $p\in\kt^*$, if there exists a solution of equation~(\ref{KRStraceequation}) then $a$ is the unique linear function on $\kt^*$ such that
  \begin{flalign}\label{condSOLITONvector}
 && \hfill\int_{\Delta} e^{-2a}f \vol=\int_{\Delta} e^{-2a}f(p) \vol &&\forall f\in\mbox{Aff}(\kt^*,\bR)
  \end{flalign} where $\vol$ is any Euclidean volume form.
\end{lemma}

\begin{rem}\label{remDONPROOF} Wang--Zhu ~\cite{WZ:krsontoric} showed the existence of a solution of equation \eqref{eq:KRsolitonPOT} while Tian--Zhu showed the uniqueness in~\cite{TZ}. Donaldson gave an alternative proof of these results in~\cite{Do:withlargesymmetry} by translating the problem into the language of labelled polytopes (in his work, the labelling is encoded in a measure on the boundary). Even though Donaldson's proof of existence relies on the higher estimates of Tian--Zhu, his proof of uniqueness holds for any labelled polytope (rational or not) and in particular implies the uniqueness of toric K\"ahler--Ricci solitons on orbifolds.
\end{rem}


\section{Generalized K\"ahler--Ricci solitons on quadrilaterals}\label{sectRESOL}

\subsection{Product case}\label{sectionProduct}

Any parallelogram is equivalent (by an affine transform) to a rectangle
$[\alpha_1,\alpha_2]\times[\beta_1,\beta_2]\subset\bR^2$, with
$\alpha_2>\alpha_1$ and $\beta_2>\beta_1$.

Let $\Delta$ be the rectangle $[\alpha_1,\alpha_2]\times[\beta_1,\beta_2]\subset\bR^2$. The normals of $\Delta$ can be written as:
\begin{align} \label{labelParallelogram}u_{\alpha_1}=C_{\alpha_1}\begin{pmatrix}
1\\
0\end{pmatrix}, \; u_{\alpha_2}=C_{\alpha_2}\begin{pmatrix}
1\\
0\end{pmatrix}, \;u_{\beta_1}=C_{\beta_1}\begin{pmatrix}
0\\
-1\end{pmatrix},\; u_{\beta_2}=C_{\beta_2}\begin{pmatrix}
0\\
-1\end{pmatrix}\end{align}
with $C_{\alpha_1}$, $C_{\beta_2}>0$ and $C_{\alpha_2}$, $C_{\beta_1}<0$.

Let $(M,\omega,J,g, T,\mu)$ be a compact, connected, K\"ahler toric $4$--orbifold. If it admits a non-trivial Hamiltonian $2$--form of order $0$ then $M$ is a product (or a finite quotient of a product) equipped with a product K\"ahler toric structure, and thus, $M$ is the product of two weighted projective spaces, \cite[Prop. 4.3]{TGQ}. The moment polytope is then a parallelogram and we identify $T\simeq S^1\times S^1$ so that $\mu=(x,y)$ and the K\"ahler metric is given by $2$ functions defined respectively on intervals $\im x = [\alpha_1,\alpha_2]$, $\im =[\beta_1,\beta_2]$. We consider the action-angle coordinates $(\mu_1=x,\mu_2=y,t,s)$ on $\mathring{M}$.

\begin{proposition}\cite[Prop. 4.3]{TGQ} \label{defnPRODlocalMet} Let $(M,\omega,J,g, T,\mu)$ be a compact, connected, K\"ahler toric $4$--orbifold with a non-trivial Hamiltonian $2$--form of order $0$. There exist functions, $A\in C^{\infty}([\alpha_1,\alpha_2])$ and $B\in C^{\infty}([\beta_1,\beta_2])$, such that $A(x)$ and $B(y)$ are positive on $\mathring{M}$,\begin{align}
g_{|_{\mathring{M}}} = \frac{dx^2}{A(x)} + & \frac{dy^2}{B(y)}  + A(x)dt^2 + B(y)ds^2 \label{PRODtoricmetric}
\end{align} and \begin{align}\label{eq:CondCompactPROD}
\begin{split}
A(\alpha_i)=0, &\;\; B(\beta_i)=0 \\
A'(\alpha_i) =\frac{2}{C_{\alpha_i}}, &\;\; B'(\beta_i) =\frac{-2}{C_{\beta_i}}.
\end{split}
\end{align}

Conversely, for any smooth functions, $A$, $B$, respectively positive on $(\alpha_1,\alpha_2)$ and $(\beta_1,\beta_2)$ and satisfying~(\ref{eq:CondCompactPROD}), the formula~(\ref{PRODtoricmetric}) defines a smooth $\omega$--compatible product toric metric on $M$.
\end{proposition}
 By using Abreu's formula~\eqref{abreuFormula}, we compute the scalar curvature of such a metric to be \begin{equation}\label{curvaturePROD}
Scal=-(A''(x)+ B''(y)).
\end{equation} In particular,
\begin{equation}\label{AverageCurvaturePROD}
  \overline{Scal}= \frac{1}{\alpha_2-\alpha_1}\left(\frac{2}{C_{\alpha_1}}- \frac{2}{C_{\alpha_2}}\right) - \frac{1}{\beta_2-\beta_1}\left(\frac{2}{C_{\beta_1}}- \frac{2}{C_{\beta_2}}\right).
\end{equation}

Writing \eqref{KRStraceequation} in this context leads to the following lemma.
\begin{lemma} Suppose that $g$ is a product toric metric corresponding to the functions $A\in C^{\infty}([\alpha_1,\alpha_2])$ and $B\in C^{\infty}([\beta_1,\beta_2])$. For $a\in \kt$, write $\langle \mu,a\rangle= a_1x + a_2y$. Then $(g,a)$ is solution of~(\ref{KRStraceequation}) if and only if there is a constant $m\in\bR$ such that
\begin{equation}\label{KRStraceequationPROD2}
  A''(x)- 2a_1 A'(x) = m-\overline{Scal}  \;\mbox{ and }\; B''(y)-2a_2 B'(y) = -m
\end{equation}
are satisfied for all $x\in (\alpha_1,\alpha_2)$ and $y\in(\beta_1,\beta_2)$.
\end{lemma}
The boundary condition~\eqref{eq:CondCompactPROD} implies that
$$ m-\overline{Scal} = \frac{-1}{\alpha_2-\alpha_1}\left(\frac{2}{C_{\alpha_1}}- \frac{2}{C_{\alpha_2}}\right) \;\mbox{ and }\;
 m = \frac{-1}{\beta_2-\beta_1}\left(\frac{2}{C_{\beta_1}}- \frac{2}{C_{\beta_2}}\right).$$

Solving the equations \eqref{KRStraceequationPROD2} using integrating factors and the boundary condition~\eqref{eq:CondCompactPROD}, gives
\begin{equation}\label{solnPROD2A}
A(x) = e^{2a_1x} \int_{\alpha_1}^{x}e^{-2a_1t}f_A(t)dt,
\end{equation}
and
\begin{equation}\label{solnPROD2B}
B(y) = e^{2a_2y} \int_{\beta_1}^{y}e^{-2a_2t}f_B(t)dt,
\end{equation} where
\begin{equation}\label{fAfBPROD2}\begin{split}
  &f_A(x)= \frac{-x}{\alpha_2-\alpha_1}\left(\frac{2}{C_{\alpha_1}}- \frac{2}{C_{\alpha_2}}\right) + \frac{2}{C_{\alpha_1}} + \frac{\alpha_1}{\alpha_2-\alpha_1}\left(\frac{2}{C_{\alpha_1}}- \frac{2}{C_{\alpha_2}}\right),\\
  &f_B(y)=\frac{y}{\beta_2-\beta_1}\left(\frac{2}{C_{\beta_1}}- \frac{2}{C_{\beta_2}}\right) - \frac{2}{C_{\beta_1}} - \frac{\beta_1}{\beta_2-\beta_1}\left(\frac{2}{C_{\beta_1}}- \frac{2}{C_{\beta_2}}\right).
  \end{split}
\end{equation}
finally, $a_1$ and $a_2$ are the unique solutions of the equations
\begin{equation}\label{krsaPROD} \int_{\alpha_1}^{\alpha_2}e^{-2a_1x}f_A(x)dx =0
    \;\mbox{ and }\; \int_{\beta_1}^{\beta_2}e^{-2a_2y}f_B(y)dy=0.\end{equation}

 \begin{rem}\label{remPROD}
   The fact there is exactly one solution for each of these equations follows from considering the functions
  $\Phi_i(a_i)=\int_{\alpha_1}^{\alpha_2}e^{-2a_i(x-x_i)}dx$ for $i=1,2$ with
  $x_1$ the root of $f_A$ and and $x_2$ the root of $f_B$. We easily show that
  that $\Phi_i$ has only one critical point since $x_1\in(\alpha_1,\alpha_2)$
  and $x_2\in(\beta_1,\beta_2)$.
 \end{rem}
 It is easy to verify that $A(x)$ and $B(y)$ as defined by~\eqref{solnPROD2A}
  and~\eqref{solnPROD2B} are positive for $x \in (\alpha_{1}, \alpha_{2})$ and $y \in (\beta_{1}, \beta_{2})$.
In conclusion, we have

\begin{lemma}\label{productlemma}
  For any labelled parallelogram $(\Delta,u)$, $\Delta\subset\kt^*$,
  there exists a generalized K\"ahler--Ricci soliton $(g,a)$ where $g$ is a
  product metric explicitly given by functions \eqref{solnPROD2A} and
  \eqref{solnPROD2B} and where $a\in \kt$ is implicitly determined by the data $(\Delta,u)$ (via equations \eqref{krsaPROD} after identifying $\Delta$ with a rectangle). \end{lemma}

\subsection{Calabi case}\label{sectionCALABI}

Any trapezoid is equivalent (by an affine transform) to a {\it Calabi trapezoid} where

\begin{definition} \label{defnCalabiPolyt}
A {\it Calabi trapezoid} is a polytope in $\bR^2$ which is the image of a rectangle $[\alpha_1,\alpha_2]\times[\beta_1,\beta_2]\subset\bR^2$, with $\alpha_1>0$ and $\beta_1\geq0$, under the map $\sigma : (x,y) \mapsto (x,xy)$. \end{definition}

Let $\Delta$ be a Calabi trapezoid with parameters $\alpha_1,\alpha_2,\beta_1,\beta_2$, with $\alpha_1>0$ and $\beta_1\geq 0$. The normals of $\Delta$ can be written as:
\begin{align} \label{labelCalabi}u_{\alpha_1}=C_{\alpha_1}\begin{pmatrix}
\alpha_1\\
0\end{pmatrix}, \; u_{\alpha_2}=C_{\alpha_2}\begin{pmatrix}
\alpha_2\\
0\end{pmatrix}, \;u_{\beta_1}=C_{\beta_1}\begin{pmatrix}
\beta_1\\
-1\end{pmatrix},\; u_{\beta_2}=C_{\beta_2}\begin{pmatrix}
\beta_2\\
-1\end{pmatrix}\end{align}
with $C_{\alpha_1}$, $C_{\beta_2}>0$ and $C_{\alpha_2}$, $C_{\beta_1}<0$. Thus, any labelled Calabi trapezoid determines and is determined by a $8$--tuple $(\alpha_1,\alpha_2,\beta_1,\beta_2,C_{\alpha_1}, C_{\alpha_2}, C_{\beta_1},C_{\beta_2})$ we shall refer to as {\it Calabi parameters}.\\

\begin{definition} \label{defnCALABItoric}
Let $(M,\omega,J,g, T,\mu)$ be a compact, connected,
K\"ahler toric $4$--orbifold. It is {\it Calabi toric} if there exist smooth $T$--invariant functions $x$ and $y\in C^{\infty}(M)$ with $x>0$, $y>0$ $g$--orthogonal gradients on $\mathring{M}$ and an identification between $\mathfrak{t}^*$ and $\bR^2$ through which the moment map is $\mu=(x,xy)$. We call $x$,$y$ the {\it Calabi coordinates}.
\end{definition}
%

For now on, we let $(M,\omega,g,\mu,T)$ be a Calabi toric $4$--orbifold with
Calabi coordinates $x$,$y$ and Calabi parameters
$(\alpha_1,\alpha_2,\beta_1,\beta_2,C_{\alpha_1}, C_{\alpha_2},
C_{\beta_1},C_{\beta_2})$. Moreover, we consider the action-angle coordinates
$(\mu_1=x,\mu_2=xy,t,s)$ on $\mathring{M}$.

\begin{proposition}\cite[Prop. 4.4]{TGQ} \label{defnCALABIlocalMet}  Let $(M,\omega,J,g, T,\mu)$ be a compact, connected, K\"ahler toric $4$--orbifold with a Hamiltonian $2$--form of order $1$. There exist functions, $A\in C^{\infty}([\alpha_1,\alpha_2])$ and $B\in C^{\infty}([\beta_1,\beta_2])$, such that $A(x)$ and $B(y)$ are positive on $\mathring{M}$,\begin{align}
g_{|_{\mathring{M}}} = x\frac{dx^2}{A(x)} + & x\frac{dy^2}{B(y)}  + \frac{A(x)}{x}(dt + yds)^2 + xB(y)ds^2 \label{CALABItoricmetric}
\end{align} and \begin{align}\label{eq:CondCompactCalabi}
\begin{split}
A(\alpha_i)=0, &\;\; B(\beta_i)=0 \\
A'(\alpha_i) =\frac{2}{C_{\alpha_i}}, &\;\; B'(\beta_i) =-\frac{2}{C_{\beta_i}}.
\end{split}
\end{align}

Conversely, for any smooth functions, $A$, $B$, respectively positive on $(\alpha_1,\alpha_2)$ and $(\beta_1,\beta_2)$ and satisfying~(\ref{eq:CondCompactCalabi}), the formula~(\ref{CALABItoricmetric}) defines a smooth $\omega$--compatible Calabi toric metric on $M$ with Calabi coordinates $x$, $y$.
\end{proposition}
 By using Abreu's formula~\eqref{abreuFormula}, we compute the scalar curvature of such a metric to be \begin{equation}\label{curvatureCALABI}
Scal=-\frac{A''(x)+ B''(y)}{x}.
\end{equation} In particular,
\begin{equation}\label{AverageCurvatureCALABI}\begin{split}
  \overline{Scal}= &\frac{2}{\alpha_2+\alpha_1}\left(\frac{A'(\alpha_1)-A'(\alpha_2)}{\alpha_2-\alpha_1} + \frac{B'(\beta_1)-B'(\beta_2)}{\beta_2-\beta_1}\right) \\
  &=\frac{4}{\alpha_2+\alpha_1}\left(\frac{1}{\alpha_2-\alpha_1}\left(\frac{1}{C_{\alpha_1}}- \frac{1}{C_{\alpha_2}}\right) - \frac{1}{\beta_2-\beta_1}\left(\frac{1}{C_{\beta_1}}- \frac{1}{C_{\beta_2}}\right) \right).
\end{split}
\end{equation}

\begin{lemma}\label{lemCalabiEQUIP} Suppose that $g$ is a Calabi toric metric corresponding to the functions $A\in C^{\infty}([\alpha_1,\alpha_2])$ and $B\in C^{\infty}([\beta_1,\beta_2])$. For $a\in \kt$, write $\langle \mu,a\rangle= a_1x + a_2xy$. Then $(g,a)$ is solution of~(\ref{KRStraceequation}) if and only if $a_2 =0$ and there exist a constant $m$ such that
\begin{equation}\label{KRStraceequationCALABI2}
  -A''(x)+2a_1 A'(x) -x\overline{Scal} = m \;\mbox{ and }\; B''(y)=m
\end{equation}
are satisfied for all $x\in (\alpha_1,\alpha_2)$ and $y\in(\beta_1,\beta_2)$.
\end{lemma}

\begin{proof} First, we will prove that $(g,a)$ is solution of~(\ref{KRStraceequation}) if and only if
\begin{equation}\label{KRStraceequationCALABI}
  -A''(x) -B''(y) -x\overline{Scal} = -2a_1 A'(x) -2a_2(yA'(x)+ xB'(y)).
\end{equation}
Notice that the $S^2\kt^*$--valued function $\bfH$ associated to $g$ is \begin{align} \bfH_{A,B} = \frac{1}{x}\begin{pmatrix} A(x) & yA(x)\\
yA(x) & x^2B(y) + y^2A(x)
\end{pmatrix}.\label{bfHABcalabi}
\end{align} Then,~(\ref{KRStraceequationCALABI}) follows from~(\ref{curvatureCALABI}) and computing $$\Delta^g\mu_1 =\Delta^gx = -\frac{A'(x)}{x}$$ and $$\Delta^g\mu_2=\Delta^gxy= -\frac{yA'(x)+ xB'(y)}{x}.$$
Now, by using~(\ref{KRStraceequationCALABI}) we have for any $\alpha_1 <x <\alpha_2$
\begin{equation*}
\begin{split}
  A'(\alpha_1)-A'(x) & = - \int_{\alpha_1}^x A''(t)dt \\
  = & (x\!-\!\alpha_1)B''(y) + \frac{x^2\!-\!\alpha_1^2}{2}\overline{Scal} -2a_1A(x) -2a_2(yA(x)+\frac{x^2\!-\!\alpha_1^2}{2}B'(y)) +f(y)
\end{split}\end{equation*}
 where $f$ is an unknown function of $y$. Since the left hand side is independent of $y$ we get $$0=\frac{\partial^2}{\partial x^2}\frac{\partial^2}{\partial y^2} A'(x) =  -2a_2B'''(y).$$ Suppose for contradiction that $a_2\neq0$. Then $B'''(y)= 0$ and hence \eqref{eq:CondCompactCalabi} implies $B(y)=\frac{m}{2}(y-\beta_1)(y-\beta_2)$ with $m<0$. Inserting this into~(\ref{KRStraceequationCALABI}) we get
\begin{equation*}
  -A''(x) -m -x\overline{Scal} = -2a_1 A'(x) -2a_2(yA'(x)+ \frac{x m}{2}( 2y-(\beta_1+\beta_2))).
\end{equation*}
Taking the derivative with respect to $y$ we obtain
$$0= 2a_2 (A'(x) +mx)$$ which is impossible, considering \eqref{eq:CondCompactCalabi}, unless $a_2=0$.\end{proof}

\begin{corollary}\label{coroEquipoisedCalabi} Let $(M,\omega,g,\mu,T)$ be a Calabi toric $4$--orbifold. If there exists $a\in \kt$ such that $(g,a)$ is a generalized K\"ahler--Ricci soliton then $a$ is equipoised as a linear function on the moment polytope.
\end{corollary}

Assume that $A(x)$ and $B(y)$ satisfy the equations~(\ref{KRStraceequationCALABI2}), as well as~(\ref{eq:CondCompactCalabi}), with respect to some $a_1\in\bR$. The value of $m$ follows from the first order boundary conditions of~(\ref{eq:CondCompactCalabi}): \begin{equation}
  \label{mCALABI}
  m=\frac{1}{\beta_2-\beta_1} \left(\frac{2}{C_{\beta_1}}- \frac{2}{C_{\beta_2}}\right).
\end{equation}
Moreover, we get
\begin{equation}
  \label{fctBcalabi} B(y)=\frac{m}{2}(y-\beta_1)(y-\beta_2)
\end{equation} and thus
\begin{equation}
  \label{condCALABI} C_{\beta_2}=-C_{\beta_1}
\end{equation} and $A$ satisfies the equation
\begin{equation}
  \label{equationfirstorderAcalabi} A'(x) -2a_1A(x) = f(x)
\end{equation} where $f(x)=C-\frac{\overline{Scal}x^2}{2} -mx$ for a constant $C$. The integrating factor method gives
\begin{equation}
  \label{solAcalabi} A(x) = e^{2a_1x} \int_{\alpha_1}^x e^{-2a_1s} f(s) ds
\end{equation}
 The constant $C$ is determined by the boundary
 conditions~(\ref{eq:CondCompactCalabi}) as the two conditions $$\frac{2}{C_{\alpha_1}}=A'(\alpha_1) -2a_1A(\alpha_1)\;\;\mbox{ and }\;\;\frac{2}{C_{\alpha_2}}=A'(\alpha_2) -2a_1A(\alpha_2)$$ give two expressions for $C$ via~(\ref{equationfirstorderAcalabi}). These expressions coincide: $$C= \frac{2}{C_{\alpha_1}} + \frac{\overline{Scal}\alpha_1^2}{2} +m\alpha_1=  \frac{2}{C_{\alpha_2}} + \frac{\overline{Scal}\alpha_2^2}{2} +m\alpha_2$$ in view of the definitions of $\overline{Scal}$ and $m$, see~(\ref{AverageCurvatureCALABI}) and~(\ref{mCALABI}). Finally, the condition $A(\alpha_2)=0$ holds if and only if

\begin{equation}
  \label{condEXISTcalabi} \int_{\alpha_1}^{\alpha_2} e^{-2a_1s} f(s) ds=0.
 \end{equation}

\begin{rem}\label{signA} By computing the integral in~(\ref{solAcalabi}), we see that $A$
    is the difference between a polynomial of degree $2$ and an exponential
    function. In particular, if $A$, defined by~(\ref{solAcalabi}), satisfies
    the boundary conditions~(\ref{eq:CondCompactCalabi}) then $A$ is positive
    on the interior of the interval $[\alpha_1,\alpha_2]$ since it is positive
    near the boundary.
\end{rem}
Using essentially the same argument of Remark~\ref{remPROD} there exists a unique real number $a_1$ solving~(\ref{condEXISTcalabi}).

In conclusion, we get :

\begin{lemma} \label{Calabisolution} Lets $(\Delta,u)$ be a labelled Calabi
    trapezoid with parameters \newline
    $(\alpha_1,\alpha_2,\beta_1,\beta_2,C_{\alpha_1}, C_{\alpha_2}, C_{\beta_1},
    C_{\beta_2})$ such that $C_{\beta_1} = -C_{\beta_2}$. Then the
    functions $A$ given by \eqref{solAcalabi} and $B$ given by \eqref{fctBcalabi} are respectively positive on $(\alpha_1,\alpha_2)$ and
    $(\beta_1,\beta_2)$, satisfying equations~(\ref{KRStraceequationCALABI2})
    and the boundary conditions~(\ref{eq:CondCompactCalabi}) where $a_1$ is the unique solution of equation~(\ref{condEXISTcalabi}).
\end{lemma}

 The linear condition~\eqref{condCALABI} on the normals is equivalent to the condition that the extremal affine function (see Section \ref{sectBG}) of the labelled trapezoid is equipoised~\cite{TGQ}. We obtain then an extension of \cite[Theorem 1, page 7]{H2FIV} in the $4$--dimensional toric orbifold case:

\begin{proposition}\label{ExistSOLcalabi}
Let $(M,\omega,T)$ be a
$4$--dimensional toric orbifold whose moment polytope is a trapezoid and
whose extremal affine function is equipoised. There exists a compatible
generalized K\"ahler--Ricci solition $(g,a)$ such that $a\in \Lie T$ is
equipoised and $g$ is of Calabi type. Moreover, up to a symplectomorphism, $g$ is explicitly
given in terms of two functions $A$ and $B$ (as usual for Calabi type metrics)
which are them self explicitly given in terms the vector $a$. In the case where $(M,\omega,T)$ is
monotone, $(g,a)$ is a K\"ahler--Ricci soliton.
\end{proposition}

Note that the vector $a$ is only implicitly known.
\begin{rem}
  All the examples of toric symplectic orbifolds with a compatible generalized K\"ahler--Ricci soliton concerned by Proposition~\ref{ExistSOLcalabi} admit also a compatible extremal K\"ahler metric of Calabi type. Indeed, along the line of the proof of \cite[Theorem 1.4]{TGQ} one can prove that every labelled trapezoids with equipoised extremal affine function are $K$- stable.
\end{rem}

\subsection{Orthotoric case}\label{sectionOrtho}

Any generic (no parallel edges) quadrilateral is equivalent to an
{\it orthotoric quadrilateral} where
\begin{definition} \label{defnOrthoPolyt}
A {\it orthotoric quadrilateral} is a quadrilateral in $\bR^2$ which is
the image of a rectangle $[\alpha_1,\alpha_2]\times[\beta_1,\beta_2]
\subset\bR^2$, with $\beta_{2} < \alpha_{1}$,
by the map $\sigma : (x,y) \mapsto (x+y,xy)$. \end{definition}

Let $\Delta$ be an orthotoric quadrilateral with parameters
$\alpha_1,\alpha_2,\beta_1,\beta_2$, with $\beta_{2} < \alpha_{1}$.
The normals of $\Delta$ can be written as:
\begin{align} \label{labelortho}u_{\alpha_1}=C_{\alpha_1}\begin{pmatrix}
\alpha_{1}\\
-1\end{pmatrix}, \; u_{\alpha_2}=C_{\alpha_2}\begin{pmatrix}
\alpha_{2}\\
-1\end{pmatrix}, \;u_{\beta_1}=C_{\beta_1}\begin{pmatrix}
\beta_1\\
-1\end{pmatrix},\; u_{\beta_2}=C_{\beta_2}\begin{pmatrix}
\beta_2\\
-1\end{pmatrix}\end{align}
with $C_{\alpha_1}$, $C_{\beta_2}>0$ and $C_{\alpha_2}$, $C_{\beta_1}<0$.
Thus, any labelled orthotoric quadrilateral determines and is determined by a
$8$--tuple
$(\alpha_1,\alpha_2,\beta_1,\beta_2,C_{\alpha_1}, C_{\alpha_2},
C_{\beta_1},C_{\beta_2})$ we shall refer to as {\it orthotoric parameters}.\\

\begin{definition} \label{defnorthotoric}
Let $(M,\omega,J,g, T,\mu)$ be a compact, connected, K\"ahler toric
$4$--orbifold. It is {\it orthotoric} if there exist smooth $T$--invariant
functions $x$ and $y\in C^{\infty}(M)$ with $x \geq y$
$g$--orthogonal gradients on $\mathring{M}$ and an
identification between $\mathfrak{t}^*$ and $\bR^2$ through which
the moment map is $\mu=(x+y,xy)$. We call $x$,$y$ the
{\it orthotoric coordinates}.
\end{definition}
\begin{rem}
When the moment polytope of $(M,\omega,J,g, T,\mu)$ is a
quadrilateral, we have $x > y$.
\end{rem}

For now on, we let $(M,\omega,g,\mu,T)$ be an orthotoric $4$--orbifold with
orthotoric coordinates $x>y$ and orthotoric parameters
$(\alpha_1,\alpha_2,\beta_1,\beta_2,C_{\alpha_1}, C_{\alpha_2},
C_{\beta_1},C_{\beta_2})$. We consider the action-angle
coordinates $(\mu_1=x+y,\mu_2=xy,t,s)$ on $\mathring{M}$.

\begin{proposition}\cite[Prop. 8]{wsd} \label{defnorthotoriclocalMet} Let $(M,\omega,J,g, T,\mu)$ be a compact, connected, K\"ahler toric $4$--orbifold with a Hamiltonian $2$--form of order $2$. There exist functions, $A\in C^{\infty}([\alpha_1,\alpha_2])$ and
$B\in C^{\infty}([\beta_1,\beta_2])$, such that $A(x)$ and $B(y)$ are
positive on $\mathring{M}$,\begin{align}
g_{|_{\mathring{M}}} = \frac{(x-y)}{A(x)}dx^2 + & \frac{(x-y)}{B(y)}dy^2
+ \frac{A(x)}{(x-y)}(dt + yds)^2 + \frac{B(y)}{(x-y)}(dt + xds)^2
\label{orthotoricmetric}
\end{align} and \begin{align}\label{eq:CondCompactorthotoric}
\begin{split}
A(\alpha_i)=0, &\;\; B(\beta_i)=0 \\
A'(\alpha_i) =\frac{2}{C_{\alpha_i}}, &\;\; B'(\beta_i) =-\frac{2}{C_{\beta_i}}.
\end{split}
\end{align}

Conversely, for any smooth functions, $A$, $B$, respectively positive
on $(\alpha_1,\alpha_2)$ and $(\beta_1,\beta_2)$ and
satisfying~(\ref{eq:CondCompactorthotoric}), the
formula~(\ref{orthotoricmetric}) defines a smooth
$\omega$--compatible orthotoric metric on $M$ with orthotoric coordinates $x$, $y$.
\end{proposition}

 By using Abreu's formula~\eqref{abreuFormula}, we compute the scalar curvature of such a metric to be
\begin{equation}\label{curvatureorthotoric}
Scal=-\frac{A''(x)+ B''(y)}{x-y}.
\end{equation} In particular,
\begin{equation}\label{AverageCurvatureorthotoric}\begin{split}
  \overline{Scal}= &\frac{2}{\alpha_2+\alpha_1-\beta_{1}-\beta_{2}}
  \left(\frac{A'(\alpha_1)-A'(\alpha_2)}{\alpha_2-\alpha_1} +
  \frac{B'(\beta_1)-B'(\beta_2)}{\beta_2-\beta_1}\right)\\
  =&\frac{4}{\alpha_2+\alpha_1-\beta_{1}-\beta_{2}}\left(\frac{1}
  {\alpha_2-\alpha_1}\left(\frac{1}{C_{\alpha_1}}- \frac{1}{C_{\alpha_2}}\right)
  - \frac{1}{\beta_2-\beta_1}\left(\frac{1}{C_{\beta_1}}- \frac{1}{C_{\beta_2}}
  \right) \right).
\end{split}
\end{equation}

\begin{lemma} Suppose that $g$ is an orthotoric metric corresponding to the functions $A\in C^{\infty}([\alpha_1,\alpha_2])$ and
    $B\in C^{\infty}([\beta_1,\beta_2])$. For $a\in \kt$, write $\langle \mu,a\rangle= a_1(x+y) + a_2xy$. Then $(g,a)$ is solution of~(\ref{KRStraceequation}) if and only if $a_2 =0$ and there exist a constant $m$ such that
\begin{equation}\label{KRStraceequationorthotoric2}
A''(x)- 2a_1 A'(x) + x\overline{Scal} = m \;\mbox{ and }\; B''(y) - 2
a_{1} B'(y)  - y\overline{Scal}=-m
\end{equation}
are satisfied for all $x\in (\alpha_1,\alpha_2)$ and $y\in(\beta_1,\beta_2)$.
\end{lemma}

\begin{proof} First, we will prove that $(g,a)$ is solution
of~(\ref{KRStraceequation}) if and only if
\begin{equation}\label{KRStraceequationorthotoric}
  A''(x) = \overline{Scal}(y-x) - B''(y) + (2 a_{1}+ 2 a_{2}y) A'(x) +
  (2a_{1}+2 a_{2}x)B'(y)
\end{equation}
is satisfied for
any $\alpha_1 <x <\alpha_2$ and any $\beta_{1}<y<\beta_{2}$.
Notice that
the $S^2\kt^*$--valued function $\bfH$ associated to $g$
is \begin{align} \bfH_{A,B} = \frac{1}{x-y}\begin{pmatrix} A(x)+B(y) &
yA(x)+xB(y)\\
yA(x) +xB(y)&  y^2A(x) +x^2B(y)
\end{pmatrix}.\label{bfHABorthotoric}
\end{align} Then,~(\ref{KRStraceequationorthotoric}) follows from
(\ref{curvatureorthotoric}) and computing
$$\Delta^g\mu_1 =\Delta^g(x+y) = -\frac{A'(x)+B'(y)}{x-y}$$ and
$$\Delta^g\mu_2=\Delta^gxy= -\frac{yA'(x)+ xB'(y)}{x-y}.$$
Since
the left hand side of (\ref{KRStraceequationorthotoric}) is independent of
$y$, we must have that
\[
\frac{\partial^{3}}{(\partial x)(\partial y)^{2}}
[\overline{Scal}(y-x) - B''(y) + (2a_{1}+2a_{2}y) A'(x) +
(2a_{1}+2a_{2}x)B'(y)] = 0
\]
or
\[
2a_{2} B'''(y) =0.
\]
Suppose for contradiction that $a_{2}\neq 0$. Then $B'''(y) = 0$ and
hence \eqref{eq:CondCompactorthotoric} implies that $B(y) =
K(y-\beta_{1})(y-\beta_{2})$ for a certain negative constant $K$.
Inserting this into \eqref{KRStraceequationorthotoric} we get
\[
A''(x) = \overline{Scal}(y-x) - 2K + (2a_{1}+2a_{2}y) A'(x) +
K(2a_{1}+2a_{2}x)(2y-\beta_{1}-\beta_{2}).
\]
Since, still, the left hand side of this equation is independent of
$y$ we get that
\[
\frac{\partial^{2}}{\partial x\partial y}
[\overline{Scal}(y-x) - 2K + (2a_{1}+2a_{2}y) A'(x) +
K(2a_{1}+2a_{2}x)(2y-\beta_{1}-\beta_{2})] = 0
\]
or
\[
a_{2}(A''(x) + 2 K) =0.
\]
Since by assumption $a_{2} \neq 0$, this implies that $A''(x) = -2K$. In
particular, $A''(x) >0$. This contradicts \eqref{eq:CondCompactorthotoric}.
Since $a_{2}\neq 0$ leads to a contradiction, the
lemma is proved.
\end{proof}

\begin{corollary}\label{coroEquipoisedorthotoric}
    Let $(M,\omega,g,\mu,T)$ be an orthotoric $4$--orbifold whose
    moment polytope is a quadrilateral (generic, necessarily). If there exists
    $a\in \kt$ such that $(g,a)$ is a generalized K\"ahler--Ricci soliton then
    $a$ is equipoised as a linear function on the moment polytope.
\end{corollary}

Assuming that $A(x)$ and $B(y)$ satisfy the equations in
\eqref{KRStraceequationorthotoric2},as well as
\eqref{eq:CondCompactorthotoric}, for some $a_{1}$

we get that
$$
A'(x) - 2a_{1} A(x) = f_{A}(x),
$$
and
$$
B'(y) - 2a_{1} B(y) = f_{B}(y),
$$
where
$$ f_{A}(x) =\frac{-\overline{Scal}}{2}(x^{2}-\alpha_{1}^{2}) + m
(x-\alpha_{1}) +
\frac{2}{C_{\alpha_{1}}},$$
$$ f_{B}(y) = \frac{\overline{Scal}}{2} (y^{2}-\beta_{1}^{2}) -
m(y-\beta_{1}) -
\frac{2}{C_{\beta_{1}}},$$
and
\begin{equation}\label{mOrtho}
m= \frac{2((\alpha_{2}^{2}-\alpha_{1}^{2})(\frac{1}{C_{\beta_{2}}}
- \frac{1}{C_{\beta_{1}}}) -
(\beta_{2}^{2}-\beta_{1}^{2})(\frac{1}{C_{\alpha_{2}}}
-\frac{1}{C_{\alpha_{1}}}))}
{(\alpha_{2}-\alpha_{1})(\beta_{2}-\beta_{1})
(\alpha_{1}+\alpha_{2}-\beta_{1}-\beta_{2})}.
\end{equation}
Using the integrating factor method we see that
\begin{equation}\label{fctA}
A(x) = e^{2a_{1}x}\int_{\alpha_{1}}^{x} f_{A}(t) e^{-2a_{1}t}\, dt
\end{equation}
and
\begin{equation}\label{fctB}
B(y)= e^{2a_{1}y}\int_{\beta_{1}}^{y} f_{B}(t) e^{-2a_{1}t}\, dt.
\end{equation}

By an argument similar to that in Remark
\ref{signA},
any solutions $A(x)$ and $B(y)$ as defined above, satisfying
\eqref{eq:CondCompactorthotoric},
also satisfy that $A(x) > 0$ for $x \in (\alpha_{1},\alpha_{2})$ and
$B(y) > 0$ for
$y \in (\beta_{1},\beta_{2})$.

Further, such functions solve
\eqref{KRStraceequationorthotoric2} and \eqref{eq:CondCompactorthotoric}
exactly when
\begin{equation}\label{ortho1} \int_{\alpha_{1}}^{\alpha_{2}} f_{A}(x) e^{-2a_{1}x}\,
dx=0
\end{equation}
and
\begin{equation}\label{ortho2}
    \int_{\beta_{1}}^{\beta_{2}} f_{B}(y) e^{-2a_{1}y}\, dy=0.
    \end{equation}
It is easy to see that $f_{A}(x)$ has exactly one root in the interval
$(\alpha_{1},\alpha_{2})$, $f_{A}(\alpha_{1}) > 0$, and
$f_{A}(\alpha_{2})<0$.
Likewise $f_{B}(y)$ has exactly one root in the interval $(\beta_{1},\beta_{2})$,
$f_{B}(\beta_{1}) >0$, and $f_{B}(\beta_{2})<0$.
Therefore, again by an argument similar to the one of Remark~\ref{remPROD}, for any labelled orthotoric
quadrilateral there exists a unique $a_{A}$ solving \eqref{ortho1} (with $a_{1}=a_{A}$) and a unique $a_{B}$ solving
\eqref{ortho2}. From this point of view the variable $a_1$ is overdetermined. Thus, the last obstacle for $A\in C^{\infty}([\alpha_1,\alpha_2])$ and $B\in C^{\infty}([\beta_1,\beta_2])$ (respectively defined by substituting $a_A$ for $a_1$ in \eqref{fctA} and by substituting $a_B$ for $a_1$ in \eqref{fctB}) to solve the generalized K\"ahler--Ricci soliton equation~\eqref{KRStraceequationorthotoric2} is that we need $a_{A}=a_{B}$.
%
 Unfortunately this is not necessarily the case. An
alternative approach to \eqref{ortho1} and \eqref{ortho2} is
to fix $a_{1}$ and then view them as a linear system of two equations in the variables
$\frac{1}{C_{\alpha_1}}, \frac{1}{C_{\alpha_2}}, \frac{1}{C_{\beta_1}}$,
and $\frac{1}{C_{\beta_2}}$. This point of view leads to Proposition~\ref{propCONE}.
For e.g. $a_{1}=0$ (the constant scalar curvature case), the solution set includes labelled polytopes associated to symplectic toric orbifolds
\cite{TGQ}. For $a_{1}\neq 0$, this is not so clear. This discussion will be
continued in Section~\ref{rationality}.

\section{The Monotone Case}\label{sectMONOT}
\subsection{Monotone Labelled Parallelogram}

For this subsection, we suppose that $(\Delta,u)$ is a rectangle $[\alpha_1,\alpha_2]\times[\beta_1,\beta_2]$ with parameters $C_{\alpha_1}, C_{\alpha_2}, C_{\beta_1}, C_{\beta_2}$ determining the normals, see~\S\ref{sectionProduct}. The proof of the next lemma is straightforward.
\begin{lemma}\label{monotoneCONDprodlemma} $(\Delta,u)$ is monotone if and only if
    \begin{equation}\label{monotoneCONDprod}
  \frac{1}{\alpha_2-\alpha_1} \left(\frac{1}{C_{\alpha_1}}-\frac{1}{C_{\alpha_2}}\right)=\frac{1}{\beta_2-\beta_1} \left(\frac{1}{C_{\beta_2}}-\frac{1}{C_{\beta_1}}\right).
\end{equation} In that case, the preferred point is \begin{equation}
\label{centerMASSprod}(x_o,y_o)=
\left(\frac{\alpha_1C_{\alpha_1}- \alpha_2C_{\alpha_2}}{C_{\alpha_1}-
C_{\alpha_2}}, \frac{\beta_1C_{\beta_1}- \beta_2C_{\beta_2}}
{C_{\beta_1}-C_{\beta_2}}\right).\end{equation}
\end{lemma}

Notice that $x_o$ is the root of $f_A$ and $y_o$ is the root of $f_B$ where
$f_A$ and $f_B$ are defined in~\eqref{fAfBPROD2}. In particular, the vector
$a=(a_1,a_2)$ defined by~\eqref{krsaPROD} and the one defined
by~\eqref{condSOLITONvector} is the same, see Remark~\ref{remPROD}.\\

\subsection{Monotone Calabi Trapezoid}
For this subsection, we suppose that $(\Delta,u)$ is a Calabi trapezoid with parameters $(\alpha_1,\alpha_2,\beta_1,\beta_2,C_{\alpha_1}, C_{\alpha_2}, C_{\beta_1}, C_{\beta_2})$, see~\S\ref{sectionCALABI} .
\begin{lemma}\label{monotoneCONDcalabilemma} $(\Delta,u)$ is monotone if and only if
    \begin{equation}\label{monotoneCONDcalabi}
  \frac{1}{\alpha_2-\alpha_1} \left(\frac{\alpha_2}{\alpha_1C_{\alpha_1}}-\frac{\alpha_1}{\alpha_2C_{\alpha_2}}\right)=\frac{1}{\beta_2-\beta_1} \left(\frac{1}{C_{\beta_2}}-\frac{1}{C_{\beta_1}}\right).
\end{equation} In that case, the preferred point is given by \begin{equation}
\label{centerMASScalabi}x_o(1,y_o)=
\frac{\alpha_1^2C_{\alpha_1}- \alpha_2^2C_{\alpha_2}}{\alpha_1C_{\alpha_1}-
\alpha_2C_{\alpha_2}}\left(1, \frac{\beta_1C_{\beta_1}- \beta_2C_{\beta_2}}
{C_{\beta_1}-C_{\beta_2}}\right).\end{equation}
\end{lemma}

\begin{proof}
The defining functions of $(\Delta,u)$ are $$L_1=\langle \cdot,
u_{\alpha_1}\rangle - C_{\alpha_1} \alpha_1^2,\,
L_2=\langle \cdot, u_{\alpha_2}\rangle - C_{\alpha_2} \alpha_2^2,
\,L_3=\langle \cdot, u_{\beta_1}\rangle, \,
L_4=\langle \cdot, u_{\beta_2}\rangle.$$
If $p=(x,xy)$ is the center of $(\Delta,u)$ we have
$L_1(p)=\alpha_{1}C_{\alpha_1}(x- \alpha_1)=L_2(p)= \alpha_{2}C_{\alpha_2}(x- \alpha_2)$ so that $$x=\frac{\alpha_1^2C_{\alpha_1}- \alpha_2^2C_{\alpha_2}}{\alpha_1C_{\alpha_1}- \alpha_2C_{\alpha_2}}$$ and we have $L_3(p)=C_{\beta_1}(\beta_1 x - xy ) =L_4(p)= C_{\beta_2}(\beta_2 x - xy)$ so that $$y=\frac{\beta_1C_{\beta_1}- \beta_2C_{\beta_2}}{C_{\beta_1}-C_{\beta_2}}.$$ Hence, we obtain an explicit value for the preferred point $p$ and one can compute easily that $L_3(p) =L_1(p)$ if and only if~(\ref{monotoneCONDcalabi}) holds.
\end{proof}
\begin{lemma}\label{propEquip=Calabi} The K\"ahler--Ricci soliton vector $a\in\kt$ of a monotone labelled
    Calabi polytope with preferred point $(x_o,x_{o}y_o)$ given by
    (\ref{centerMASScalabi}) is equipoised if and only
\begin{equation}
  \label{vectorSOLNcalabib0}
  C_{\beta_2}=-C_{\beta_1}.
\end{equation} In that case, $a=(a_1,0)$ where $a_1$ is the only solution of equation \eqref{condEXISTcalabi}.
\end{lemma}

\begin{proof} Writing $\mu =(\mu_1,\mu_2)=(x,xy)$ as above the condition
    (\ref{condSOLITONvector}) is equivalent to

\begin{equation*}
  \label{vectorSOLNprf1}\int_{\beta_1}^{\beta_2}\!\!\!\int_{\alpha_1}^{\alpha_2}\! \! e^{-2a_1x-2a_2y}(x-x_o)xdxdy = 0  \;\mbox{ and }\; \int_{\beta_1}^{\beta_2}\!\!\!\int_{\alpha_1}^{\alpha_2}\! \!e^{-2a_1x-2a_2y}(y-y_o)x^2dxdy = 0.
\end{equation*}
As $a$ is unique, $a_2=0$ if and only
$$0=\int_{\beta_1}^{\beta_2}(y-y_o)dy = \frac{(\beta_2-\beta_1)^2}{2}(C_{\beta_1} + C_{\beta_2})$$
and there exists $a_1\in \bR$ such that \begin{equation}\label{defnRSvectorCALABI}
\int_{\alpha_1}^{\alpha_2} \! e^{-2a_1x}(x-x_o)xdx = 0.
\end{equation}

There always exists a (unique) solution of this last equation as being the
unique critical point of the strictly convex proper function
$\Phi(a_1)= \int_{\alpha_1}^{\alpha_2} e^{-2a_1 (x-x_o)}xdx.$ In the monotone case, by using relation~\eqref{monotoneCONDcalabi} and the definitions of $\overline{Scal}$ and $m$, see~(\ref{AverageCurvatureCALABI}) and~(\ref{mCALABI}), we have

$$C= \frac{2}{C_{\alpha_1}} + \frac{\overline{Scal}\alpha_1^2}{2} +m\alpha_1 =0$$
and we have $$\frac{\overline{Scal}\,x}{2} +m =
\frac{1}{\alpha_2-\alpha_1}\left(\frac{2}{\alpha_1C_{\alpha_1}}-\frac{2}
{\alpha_2C_{\alpha_2}}\right)( x-x_o).$$ Hence, equations \eqref{condEXISTcalabi}
and \eqref{defnRSvectorCALABI} define the same value of $a_1$.\end{proof}

\subsection{Monotone Orthotoric Quadrilaterals}\label{sectMonotOrtho}

For this subsection, we suppose that $(\Delta,u)$ is an orthotoric
quadrilateral
with parameters $(\alpha_1,\alpha_2,\beta_1,\beta_2,C_{\alpha_1}, C_{\alpha_2},
C_{\beta_1}, C_{\beta_2})$, see~\S\ref{sectionOrtho}.
\begin{lemma} $(\Delta,u)$ is monotone if and only if
    \begin{equation}\label{monotoneCONDOrtho}
\frac{(\alpha_{1}-\beta_{1})(\alpha_{2}-\beta_{1})
\frac{1}{C_{\beta_2}}
- (\alpha_{1}-\beta_{2})(\alpha_{2}-\beta_{2})
\frac{1}{C_{\beta_1}}}{(\beta_{2}-\beta_{1})}
=
\frac{(\alpha_{2}-\beta_{1})(\alpha_{2}-\beta_{2})
\frac{1}{C_{\alpha_1}}-(\alpha_{1}-\beta_{1})(\alpha_{1}-\beta_{2})
\frac{1}{C_{\alpha_2}}}{(\alpha_{2}-\alpha_{1})}.
\end{equation} In that case, the preferred point, $(x_o + y_o,
x_{o}y_{o})$, is given by
\begin{equation}\label{centerMASSOrtho}
\begin{array}{rcl}
x_o + y_o & = & \frac{(\alpha_{1}^{2} -
\beta_{1}^{2})C_{\alpha_{1}}C_{\beta_{1}} -(\alpha_{2}^{2} -
\beta_{1}^{2})C_{\alpha_{2}}C_{\beta_{1}} -(\alpha_{1}^{2} -
\beta_{2}^{2})C_{\alpha_{1}}C_{\beta_{2}} +(\alpha_{2}^{2} -
\beta_{2}^{2})C_{\alpha_{2}}C_{\beta_{2}}}{(\alpha_{1} -
\beta_{1})C_{\alpha_{1}}C_{\beta_{1}} -(\alpha_{2} -
\beta_{1})C_{\alpha_{2}}C_{\beta_{1}} -(\alpha_{1} -
\beta_{2})C_{\alpha_{1}}C_{\beta_{2}} +(\alpha_{2} -
\beta_{2})C_{\alpha_{2}}C_{\beta_{2}}}\\
\\
x_{o}y_{o}& = &\frac{\alpha_{1}\beta_{1}(\alpha_{1} -
\beta_{1})C_{\alpha_{1}}C_{\beta_{1}} -\alpha_{2}\beta_{1}(\alpha_{2} -
\beta_{1})C_{\alpha_{2}}C_{\beta_{1}} -\alpha_{1}\beta_{2}(\alpha_{1} -
\beta_{2})C_{\alpha_{1}}C_{\beta_{2}} +\alpha_{2}\beta_{2}(\alpha_{2} -
\beta_{2})C_{\alpha_{2}}C_{\beta_{2}}}{(\alpha_{1} -
\beta_{1})C_{\alpha_{1}}C_{\beta_{1}} -(\alpha_{2} -
\beta_{1})C_{\alpha_{2}}C_{\beta_{1}} -(\alpha_{1} -
\beta_{2})C_{\alpha_{1}}C_{\beta_{2}} +(\alpha_{2} -
\beta_{2})C_{\alpha_{2}}C_{\beta_{2}}}.
 \end{array}
 \end{equation}
\end{lemma} Note that in both of these last expressions the denominators are the same.
\begin{proof}
The proof is completely similar to the proof of Lemma
\ref{monotoneCONDcalabilemma} with the only difference being that now
the defining functions of $(\Delta,u)$ are $L_1=\langle \cdot,
u_{\alpha_1}\rangle - C_{\alpha_1} \alpha_1^2,\,
L_2=\langle \cdot, u_{\alpha_2}\rangle - C_{\alpha_2} \alpha_2^2,
\,L_3=\langle \cdot, u_{\beta_1}\rangle - C_{\beta_1} \beta_1^2,$ and
$L_4=\langle \cdot, u_{\beta_2}\rangle  - C_{\beta_2} \beta_2^2.$
\end{proof}
\begin{lemma}\label{reconsilationOrtho}The K\"ahler--Ricci soliton vector $a\in\kt$ of a monotone labelled
    orthotoric quadrilateral is equipoised if and only if
the unique
$a_{A}$ solving \eqref{ortho1} equals the unique $a_{B}$ solving
\eqref{ortho2}. In that case, $a = (a_{A},0) = (a_{B},0)$.
\end{lemma}
\begin{proof}
Writing $\mu =(\mu_1,\mu_2)=(x+y,xy)$ as above the condition
    (\ref{condSOLITONvector}) is equivalent to the following pair of equation:

\begin{equation*}
\begin{array}{lcc}
  \int_{\beta_1}^{\beta_2}\!\!\!\int_{\alpha_1}^{\alpha_2}\! \!
  e^{-2a_1(x+y)-2a_2xy}(x+y-(x_o+y_{o}))(x-y)dxdy & = & 0, \\
\\
  \int_{\beta_1}^{\beta_2}\!\!\!\int_{\alpha_1}^{\alpha_2}\! \!
  e^{-2a_1(x+y)-2a_2xy}(xy-x_{o}y_o)(x-y)dxdy & = & 0.
  \end{array}
\end{equation*}
As $a$ is unique, $a_{2}=0$ if and only if there exists
an $a_{1}$ such that
\begin{equation*}
  \int_{\beta_1}^{\beta_2}\!\!\!\int_{\alpha_1}^{\alpha_2}\! \!
  e^{-2a_1(x+y)}(x+y-(x_o+y_{o}))(x-y)dxdy = 0  \;\mbox{ and }\;
  \int_{\beta_1}^{\beta_2}\!\!\!\int_{\alpha_1}^{\alpha_2}\! \!
  e^{-2a_1(x+y)}(xy-x_{o}y_o)(x-y)dxdy = 0.
\end{equation*}
This can be written as

  \begin{equation}\label{DonOrtho1}
  x_{o}+y_{o}= \frac{\int_{\beta_1}^{\beta_2}\!\!\!\int_{\alpha_1}^{\alpha_2}
  \! \! e^{-2a_1(x+y)}(x^{2}-y^{2})dxdy}{\int_{\beta_1}^{\beta_2}\!\!\!
  \int_{\alpha_1}^{\alpha_2}\! \!
  e^{-2a_1(x+y)}(x-y)dxdy}  \;\mbox{ and }\;
  x_{o}y_{o}=\frac{\int_{\beta_1}^{\beta_2}\!\!\!\int_{\alpha_1}^{\alpha_2}
  \! \! e^{-2a_1(x+y)}xy(x-y)dxdy}{\int_{\beta_1}^{\beta_2}\!\!\!
  \int_{\alpha_1}^{\alpha_2}\! \!
  e^{-2a_1(x+y)}(x-y)dxdy}.
\end{equation}

We will now see that under the condition of \eqref{monotoneCONDOrtho}
(and \eqref{centerMASSOrtho})
the system \eqref{DonOrtho1} is equivalent to \eqref{ortho1} and
\eqref{ortho2} (as a system).

Solving for $C_{\beta_{2}}$ in \eqref{monotoneCONDOrtho} and
substituting this into \eqref{centerMASSOrtho}, we get that
\begin{equation*}
\begin{array}{ccc}
x_o + y_o & = & \frac{(\alpha_{2}^{2} -
\alpha_{1}^{2})C_{\alpha_{1}}C_{\alpha_{2}} +(\alpha_{1}^{2} -
\beta_{1}^{2})C_{\alpha_{1}}C_{\beta_{1}} - (\alpha_{2}^{2} -
\beta_{1}^{2})C_{\alpha_{2}}C_{\beta_{1}}}
{(\alpha_{2} -
\alpha_{1})C_{\alpha_{1}}C_{\alpha_{2}} +(\alpha_{1} -
\beta_{1})C_{\alpha_{1}}C_{\beta_{1}} - (\alpha_{2} -
\beta_{1})C_{\alpha_{2}}C_{\beta_{1}}}\\
\\
x_{o}y_{o}& = &\frac{\alpha_{1}\alpha_{2}(\alpha_{2} -
\alpha_{1})C_{\alpha_{1}}C_{\alpha_{2}} +\alpha_{1}\beta_{1}(\alpha_{1} -
\beta_{1})C_{\alpha_{1}}C_{\beta_{1}} - \alpha_{2}\beta_{1}(\alpha_{2} -
\beta_{1})C_{\alpha_{2}}C_{\beta_{1}}}{(\alpha_{2} -
\alpha_{1})C_{\alpha_{1}}C_{\alpha_{2}} +(\alpha_{1} -
\beta_{1})C_{\alpha_{1}}C_{\beta_{1}} - (\alpha_{2} -
\beta_{1})C_{\alpha_{2}}C_{\beta_{1}}}.
 \end{array}
 \end{equation*}
 Since $C_{\alpha_{2}}\neq 0$ and $C_{\beta_{1}}\neq 0$, this last system of equations is equivalent to
 \begin{equation*}
     \begin{array}{ccc}
	 C_{\alpha_{2}}& = & \frac{\alpha_{1}^{2} - \alpha_{1}(x_{o}+y_{o}) +
	 x_{o}y_{o}}{\alpha_{2}^{2} - \alpha_{2}(x_{o}+y_{o}) +
	 x_{o}y_{o}}C_{\alpha_{1}}\\
	 \\
	 C_{\beta_{1}}& = & \frac{\alpha_{1}^{2} - \alpha_{1}(x_{o}+y_{o}) +
	 x_{o}y_{o}}{\beta_{1}^{2} - \beta_{1}(x_{o}+y_{o}) +
	 x_{o}y_{o}}C_{\alpha_{1}}.
	 \end{array}
	 \end{equation*}
Therefore,
\eqref{DonOrtho1} is equivalent to

\begin{equation}\label{DonOrtho2}
     \begin{array}{ccc}
	 C_{\alpha_{2}}& = & \frac{\int_{\beta_1}^{\beta_2}\!\!\!\int_{\alpha_1}^{\alpha_2}
  \! \! e^{-2a_1(x+y)}(\alpha_{1}-x)(\alpha_{1}-y)(x-y)dxdy}{\int_{\beta_1}^{\beta_2}\!\!\!\int_{\alpha_1}^{\alpha_2}
  \! \! e^{-2a_1(x+y)}(\alpha_{2}-x)(\alpha_{2}-y)(x-y)dxdy}C_{\alpha_{1}}\\
  \\
	 C_{\beta_{1}}& = &\frac{\int_{\beta_1}^{\beta_2}\!\!\!\int_{\alpha_1}^{\alpha_2}
  \! \! e^{-2a_1(x+y)}(\alpha_{1}-x)(\alpha_{1}-y)(x-y)dxdy}{\int_{\beta_1}^{\beta_2}\!\!\!\int_{\alpha_1}^{\alpha_2}
  \! \! e^{-2a_1(x+y)}(\beta_{1}-x)(\beta_{1}-y)(x-y)dxdy}C_{\alpha_{1}}.
	 \end{array}
	 \end{equation}
On the other hand, solving for $\frac{1}{C_{\beta_{2}}}$ in
\eqref{monotoneCONDOrtho} and
substituting this into \eqref{AverageCurvatureorthotoric} and
\eqref{mOrtho}, we get that
$$
\overline{Scal}=
4\left(\frac{\frac{1}{C_{\alpha_{1}}}}{(\alpha_{2}-\alpha_{1})(\alpha_{1}-\beta_{1})}
-\frac{\frac{1}{C_{\alpha_{2}}}}{(\alpha_{2}-\alpha_{1})(\alpha_{2}-\beta_{1})}
-\frac{\frac{1}{C_{\beta_{1}}}}{(\alpha_{1}-\beta_{1})(\alpha_{2}-\beta_{1})}\right)
$$
and
$$
m= 2\left(\frac{(\alpha_{2}+\beta_{1})\frac{1}{C_{\alpha_{1}}}}{(\alpha_{2}-\alpha_{1})(\alpha_{1}-\beta_{1})}
-\frac{(\alpha_{1}+\beta_{1})\frac{1}{C_{\alpha_{2}}}}{(\alpha_{2}-\alpha_{1})(\alpha_{2}-\beta_{1})}
-\frac{(\alpha_{1}+\alpha_{2})\frac{1}{C_{\beta_{1}}}}{(\alpha_{1}-\beta_{1})(\alpha_{2}-\beta_{1})}
\right).
$$
Therefore (in the monotone case), equations \eqref{ortho1} and
\eqref{ortho2} are equivalent to
\begin{equation}\label{usOrtho1}
    \begin{array}{cl}
	
&\frac{\alpha_{2}-\beta_{1}}{C_{\alpha_{1}}}\int_{\alpha_1}^{\alpha_2}
 \! \!
 e^{-2a_1x}(\alpha_{2}-x)(x-\beta_{1})\,dx\\
 \\
=&\frac{\alpha_{1}-\beta_{1}}{C_{\alpha_{2}}}\int_{\alpha_1}^{\alpha_2}
  \! \!
  e^{-2a_1x}(\alpha_{1}-x)(x-\beta_{1})\,dx+\frac{\alpha_{2}-\alpha_{1}}{C_{\beta_{1}}}\int_{\alpha_1}^{\alpha_2}
  \! \!
  e^{-2a_1x}(\alpha_{1}-x)(x-\alpha_{2})\,dx\\
  \\
   \;\mbox{ and }\;
  \\
& \frac{\alpha_{2}-\beta_{1}}{C_{\alpha_{1}}}\int_{\beta_1}^{\beta_2}
  \! \!
  e^{-2a_1y}(\alpha_{2}-y)(y-\beta_{1})\,dy\\
  \\
 = & \frac{\alpha_{1}-\beta_{1}}{C_{\alpha_{2}}}\int_{\beta_1}^{\beta_2}
  \! \!
  e^{-2a_1y}(\alpha_{1}-y)(y-\beta_{1})\,dy+ \frac{\alpha_{2}-\alpha_{1}}{C_{\beta_{1}}}\int_{\beta_1}^{\beta_2}
  \! \!
  e^{-2a_1y}(\alpha_{1}-y)(y-\alpha_{2})\,dy.

  \end{array}
   \end{equation}
Viewing \eqref{usOrtho1}
as a linear system in $\frac{1}{C_{\alpha_{2}}}$
and$\frac{1}{{C_{\beta_{1}}}}$, it is now
a simple matter to verify
that \eqref{usOrtho1} is equivalent to \eqref{DonOrtho2}.
Due to this equivalence, the lemma now follows.
\end{proof}

Lemma \ref{reconsilationOrtho} together with the results in Section
\ref{sectionOrtho} implies the following result.
\begin{proposition} \label{propMONOTequip=existORTHO}The K\"ahler--Ricci soliton vector $a\in\kt$ of a
    monotone labelled orthotoric quadrilateral is equipoised if and only if there exists a
    compatible K\"ahler--Ricci soliton $(g,a)$, such that $g$ is
    orthotoric.
    \end{proposition}

\section{Rationality/applications to orbifolds}\label{rationality}

Recall that for a labelled polytope $(\Delta,u)$, $\Delta\subset \kt^*$, to be
associated to a symplectic toric orbifold $(M,\omega,T)$ it has to be rational
with respect to a lattice $\Lambda$, that is $u\subset \Lambda$ with
$T=\kt/\Lambda$. Then, the symplectic toric orbifold is obtained via
the so-called Delzant construction, \cite{delzant:corres,LT:orbiToric}, which
also produce the moment map $\mu :M\ra \kt^*$. In this section, we are
concerned to verify whether or not the set we found of labelled quadrilaterals
$(\Delta,u)$ admitting a solution to \eqref{KRStraceequation} contains a
subset of rational labelled quadrilaterals. Apart from weighted projective
planes and their quotients, this subset exactly
corresponds to the set of generalized K\"ahler--Ricci compact toric
$4$--orbifolds with Hamiltonian $2$--forms.

\begin{rem}
   A labelled polytope can be rational with respect to different lattices
   each of them containing the lattice produced by the $\bZ$--span of the
   normals. This minimal lattice leads to a simply connected orbifold
   via the Delzant--Lerman--Tolman correspondence where inclusion of
   lattices corresponds to finite orbifold covering, see~\cite{H2FII}.
 \end{rem}

A symplectic toric orbifold is a smooth manifold if and only if its labelled polytope is {\it Delzant}, \cite{delzant:corres} in the following sense.
\begin{definition}  A labelled polytope $(\Delta,u)$ is Delzant if $\Lambda=\mbox{span}_{\bZ}\{u_1,\dots,u_d\}$ is a lattice and for each vertex $F_I =\cap_{i\in I} F_{i}$, the set $\{u_i\,|\,i\in I\}$ is a basis\footnote{Recall that the polytope is assumed to be simple meaning that if $F_I$ is a vertex, $|I|=\dim \kt$.} of $\Lambda$ where $F_1,\dots, F_d$ are the facets of $\Delta$ and $u=\{u_1,\dots,u_d\}$.
\end{definition}


If $\Delta$ is a convex quadrilateral there exists a lattice $\Lambda$
containing a set of inward normals $u$ to $\Delta$ if and only if the cross-ratio of
its normal lines is rational or infinite, see~\cite[Section 6]{TGQ}.
The condition of being a Delzant labelled convex quadrilateral is more
restrictive and can happen only if it is a trapezoid (including parallelogram)
so the Delzant construction produces the Hirzebruch surfaces, see~\cite{karshon}.

\subsection{Rational labelled Parallelogram}
Using the definitions, it is straightforward to verify the following lemma.

\begin{lemma}\label{parallelogramRATIONAL}
Let $(\Delta,u)$ be a labelled parallelogram with parameters
$$(\alpha_1,\alpha_2,\beta_1,\beta_2, C_{\alpha_1},C_{\alpha_2}, C_{\beta_1},
C_{\beta_2}).$$ $(\Delta,u)$ is rational with respect to a lattice if and only
if the $2$ following numbers $$p =-\frac{C_{\beta_2}}{C_{\beta_1}},\;
k =-\frac{C_{\alpha_2}}{C_{\alpha_1}}$$ are rational.
Moreover, $(\Delta,u)$ is Delzant if and only if $p=k=1$.
\end{lemma}

In the case where $(\Delta,u)$ is Delzant then $(\Delta,u)$ is associated
via the Delzant--Lerman--Tolman correspondence to
the product Hirzebruch surface $\bcp^1 \times \bcp^1$. It is easy to
check that this is the unique case (among rational and
non-rational cases) where ~\eqref{krsaPROD} is
satisfied for $a_{1}=a_{2}=0$.

\begin{corollary}\label{smoothparallel}
The generalized K\"ahler Ricci solitons obtained in Lemma
\ref{productlemma} are
smooth K\"ahler metrics on $\bcp^1 \times \bcp^1$ if and only if the
metrics are CSC. In that case the metrics
are simply products on $\bcp^1 \times \bcp^1$ of (re-scales of) Fubini-Study
metrics.
Such metrics are sometimes called
generalized K\"ahler-Einstein metrics.
\end{corollary}

Combining the the first part of
Lemma~\ref{parallelogramRATIONAL} with the relation
~\eqref{monotoneCONDprod} we easily get the following lemma.

\begin{lemma} \label{parallelogramRATIONALMONOTONE}
Let $(\Delta,u)$ be a rational labelled parallelogram with
parameters $$\left(\alpha_1,\alpha_2,\beta_1,\beta_2,
C_{\alpha_{1}},-kC_{\alpha_{1}} , C_{\beta_1},-pC_{\beta_1}\right)$$ where
$p$ and $k$ are positive rational numbers. $(\Delta,u)$ is monotone
if and only if
\begin{equation}\label{parallelogramrationalMONOTONE}
\frac{1}{\alpha_2-\alpha_1} \left(
\frac{1+k}{k}\right)\frac{1}{C_{\alpha_1}}=\frac{1}{\beta_1-\beta_1}
\left(\frac{1+p}
{p}\right)\frac{1}{C_{\beta_1}}.
\end{equation}
\end{lemma}

If we think of $\frac{-C_{\alpha_{1}}}{C_{\beta_{1}}}>0$ as a new parameter, we can formulate the following Corollary.
\begin{corollary}
Given any parallelogram $\Delta$, there is a family  of
inward normals $u(p,k,r)$, depending on two positive rational
parameters $p,k\in\bQ_{>0}$ and a positive real parameter
$r\in\bR_{>0}$,
such that $(\Delta,u(p,k,r))$ corresponds to a symplectic toric orbifold
admitting a compatible generalized K\"ahler--Ricci soliton given by
Lemma~\ref{productlemma}. This family contains a $2$--rational
parameter family of monotone labelled parallelograms (in which case the
metric is  a K\"ahler--Ricci soliton) and a $1$--real
parameter family of labelled parallelograms where the metric in
question is CSC. The latter family coincides with the $1$--real
parameter family of Delzant labelled parallelograms.
\end{corollary}

In the Delzant case, $p=k=1$ \eqref{parallelogramrationalMONOTONE} becomes
$$\frac{1}{\alpha_2-\alpha_1} \frac{1}{C_{\alpha_1}}=\frac{1}{\beta_1-\beta_1}
\frac{1}{C_{\beta_1}},$$
which reconciles well with Corollary \ref{smoothparallel} since
it clearly corresponds to the product metric being a product of two
Fubini-Study metrics {\em of the same curvature}. Thus the
generalized K\"ahler--Einstein metric
is K\"ahler--Einstein.

\subsection{Rational labelled Calabi trapezoids}

\begin{lemma}\label{calabiRATIONAL}
 Let $(\Delta,u)$ be a labelled Calabi trapezoid with parameters
$$(\alpha_1,\alpha_2,\beta_1,\beta_2, C_{\alpha_1},C_{\alpha_2}, C_{\beta_1},
C_{\beta_2}).$$ $(\Delta,u)$ is rational with respect to a lattice if and only
if the $3$ following numbers $$p =-\frac{C_{\beta_2}}{C_{\beta_1}},\;
k =\frac{(\beta_2 -\beta_1)C_{\beta_2}}{\alpha_1C_{\alpha_1}},\;
l =-\frac{(\beta_2 -\beta_1)C_{\beta_2}}{\alpha_2C_{\alpha_2}}$$ are rational.
Moreover, $(\Delta,u)$ is Delzant if and only if $p=1$ and $k=l\in\bN$.
\end{lemma}
The proof of the first part of this last lemma can be found in \cite{TGQ}
(and is an easy consequence of the Delzant construction) while the last part
is well known~\cite{karshon}. Via the Delzant--Lerman--Tolman correspondence,
Delzant labelled trapezoids correspond to Hirzebruch surfaces (i.e
$\bcp^1$--bundle over $\bcp^1$) and have been studied at length by from
different points of view, see e.g. \cite{gaud:hirz}. Guan constructs on
them generalized K\"ahler--Ricci solitons~\cite{guan2} and here we recover this
result as a particular case:

\begin{corollary}
 If $(\Delta,u)$ is Delzant then its extremal affine function is equipoised. In particular, in view of Proposition~\ref{ExistSOLcalabi}, every Hirzebruch surface admits a compatible generalized K\"ahler-Ricci soliton of Calabi type.
\end{corollary}
%

\begin{lemma}\label{calabiRATIONALMONOTONE}
 Let $(\Delta,u)$ be an equipoised rational labelled Calabi trapezoid with parameters $$\left(\alpha_1,\alpha_2,\beta_1,\beta_2, \frac{\beta_2-\beta_1}{k\alpha_1}C_{\beta_2},-\frac{\beta_2-\beta_1}{l\alpha_2}C_{\beta_2},-C_{\beta_2},C_{\beta_2}\right)$$ where $k$, $l$ are positive rational numbers. $(\Delta,u)$ is monotone if and only if
\begin{equation}\label{rationalMONOTONE}
 k\alpha_2 +l\alpha_1 =2(\alpha_2 -\alpha_1)
\end{equation}
If in addition, $(\Delta,u)$ is Delzant then $k=l=1$ and thus $(\Delta,u)$ is associated via the Delzant--Lerman--Tolman to the first Hirzebruch surface.
\end{lemma}
The last part of Lemma~\ref{calabiRATIONALMONOTONE} is well-known, see e.g.\cite{gaud:hirz}.

\begin{proof}[Proof of lemma~\ref{calabiRATIONALMONOTONE}]
 The first part of the lemma follows from the first part of Lemma~\ref{calabiRATIONAL} combined with the relation~\eqref{monotoneCONDcalabi}. If $(\Delta,u)$ is Delzant then $k=l$ and \eqref{rationalMONOTONE} become $$(k-2)\alpha_2 + (k+2)\alpha_1 =0.$$ Sign considerations conclude the proof. \end{proof}

\begin{corollary}
  Given any trapezoid $\Delta$, there is a family of inward normals $u(k,l)$ depending on two positive rational parameters $k,l\in\bQ_{>0}$ such that $(\Delta,u(k,l))$ corresponds to a symplectic toric orbifold admitting a compatible generalized K\"ahler--Ricci soliton given by the Proposition~\ref{ExistSOLcalabi}. The vertices of $\Delta$ are contained in a lattice if and only if there is a monotone labelled trapezoid in the family $(\Delta,u(k,l))$. In that case there, is a $1$--rational parameter family of monotone labelled trapezoids in this family.
\end{corollary}

\subsection{Rational labelled orthotoric quadrilaterals}

Let $(\Delta,u)$ be a labelled orthotoric quadrilateral with parameters $$(\alpha_1,\alpha_2,\beta_1,\beta_2, C_{\alpha_1},C_{\alpha_2}, C_{\beta_1},C_{\beta_2}),$$ in particular, $\beta_1<\beta_2<\alpha_1<\alpha_2$. We showed in~\S\ref{sectionOrtho} that there is an orthotoric solution of \eqref{KRStraceequation} on $(\Delta,u)$ if and only if there exists $a_1\in\bR$ satisfying equations~\eqref{ortho1} and~\eqref{ortho2}.

The case $a_1=0$ corresponds to orthotoric K\"ahler metrics with constant scalar
curvature (CSC). In \cite{TGQ}, it is shown that every generic quadrilateral
whose vertices lie in a lattice is the common moment polytope of a $\bQ^2$ family of CSC toric K\"ahler compact orbifolds. However, it is easy to find examples of generic quadrilateral of rational type (i.e for which there exists a rational labelling) admitting no rational labelling $u$ for which $(\Delta,u)$ admits a CSC metric. For instance, for any transcendental number $\alpha>1$ and rational number $r<0$, the orthotoric quadrilateral with parameters $\beta_1=0$, $\beta_2 = \frac{\alpha}{\alpha(1-r)+r}$, $\alpha_1=1$, $\alpha_2=\alpha$ is of rational type (see lemma~\ref{expressdesC}) but admits no rational labelling for which the extremal affine function is constant, see~\cite[Remark~6.12]{TGQ}.

Let us focus on the case $a_1\neq 0$. Notice that a real number $a_1\neq0$ satisfies equations~\eqref{ortho1} and~\eqref{ortho2} if and only if
\begin{equation}\label{polynEXPequatortho}\begin{split}
&e^{-2a_1(\alpha_2-\alpha_1)}F_A(\alpha_2,a_1)=F_A(\alpha_1,a_1),\\
&e^{-2a_1(\beta_2-\beta_1)}F_B(\beta_2,a_1)=F_B(\beta_1,a_1)
\end{split}
\end{equation} where $F_A(x,a_1)$ and $F_B(x,a_1)$ are the following polynomials of $2$ variables
\begin{equation}\begin{split}
&F_A(x,a_1)=\frac{\overline{Scal}}{8}+\frac{a_1}{4}\left( \overline{Scal}x -m\right) -\frac{a_1^2}{2}f_A(x)\\
&F_B(x,a_1)=\frac{\overline{Scal}}{8}+\frac{a_1}{4}\left( \overline{Scal}x -m\right) +\frac{a_1^2}{2}f_B(x).
\end{split}
\end{equation}
The definitions of $f_A(x)$, $f_B(x)$, $\overline{Scal}$ and $m$ are in \S~\ref{sectionOrtho}. The first and second equations of \eqref{polynEXPequatortho} are respectively $a_1^3$ and $-a_1^3$ times the integrals~\eqref{ortho1} and~\eqref{ortho2}. The assumption $a_1\neq0$ is then essential. Actually, $a_1=0$ is a solution of \eqref{polynEXPequatortho} and then there is at most two solutions to this system of equations.

Let $a_A$ be the unique solution of~\eqref{ortho1} and $a_B$ the unique solution of~\eqref{ortho2}. One can compute that $a_A=0$ (respectively $a_B=0$) if and only if respectively
\begin{equation*}
\overline{Scal}=\frac{-12}{(\alpha_2-\alpha_1)^2}\left(\frac{1}{C_{\alpha_1}}+\frac{1}{C_{\alpha_2}}\right)\;\;\mbox{ resp. }\;\; \overline{Scal}=\frac{-12}{(\beta_2-\beta_1)^2}\left(\frac{1}{C_{\beta_1}}+\frac{1}{C_{\beta_2}}\right).
\end{equation*}

 \begin{lemma}\label{expressdesC}\cite{TGQ} A labelled orthotoric quadrilateral associated to orthotoric parameters $(\alpha_1,\alpha_2,\beta_1,\beta_2,C_{\alpha_1}, C_{\alpha_2}, C_{\beta_1},C_{\beta_2})$ is rational if and only if the $4$ following numbers are rational
  $$r=\frac{(\beta_2-\alpha_1)(\alpha_2-\beta_1)}{(\beta_2- \beta_1)(\alpha_2-\alpha_1)},\; p=\frac{(\alpha_1-\beta_1)C_{\beta_1}}{(\beta_2-\alpha_1)C_{\beta_2}},\; k=\frac{(\beta_1-\alpha_2)C_{\alpha_2}}{(\beta_2-\beta_1)C_{\beta_2}},\; l=\frac{(\alpha_1-\beta_1)C_{\alpha_1}}{(\beta_2-\beta_1)C_{\beta_2}}.$$
\end{lemma}

\begin{lemma}\label{FactsF_AF_B}
 Let $(\Delta,u)$ be a rational labelled orthotoric quadrilateral with parameters $(\alpha_1,\alpha_2,\beta_1,\beta_2,C_{\alpha_1}, C_{\alpha_2}, C_{\beta_1},C_{\beta_2})$ then
\begin{itemize}
  \item[(1)] it is impossible that both the right and left hand sides of the two equations of~\eqref{polynEXPequatortho} vanish simultaneously,
  \item[(2)] if $\alpha_1,\alpha_2,\beta_1,\beta_2,a_1$ satisfy \eqref{polynEXPequatortho} with $a_1\neq 0$, then at least one of the numbers $\alpha_1,\alpha_2,\beta_1,\beta_2,a_1$ is transcendental,
  \item[(3)] if $\alpha_1,\alpha_2,\beta_1,\beta_2$ are algebraic and $\alpha_2-\alpha_1,\beta_2-\beta_1\in\bQ$, there is no non-zero real number $a_1$ satisfying~\eqref{polynEXPequatortho}.

\end{itemize}
\end{lemma}

\begin{proof}
 For $(1)$, the difference between $F_A(x,a_1)$ and $F_B(x,a_1)$, seen as
 polynomials in the variable $x$ only, is a constant:
$$F_A(x,a_1)-F_B(x,a_1) =-\frac{a_1^2}{2}\left(\frac{\overline{Scal}}{2}
(\alpha_1^2-\beta_1^2) -m(\alpha_1-\beta_1) +\frac{2}{C_ {\alpha_1}}- \frac{2}
{C_{\beta_1}} \right).$$ This implies that the sum of their roots is the same
and, thus, if $\alpha_1$ and $\alpha_2$ are the roots of $F_A(x,a_1)$,
$\beta_1$ and $\beta_2$ cannot be the roots of $F_B(x,a_1)$ since
$\alpha_1+\alpha_2\neq \beta_1+\beta_2$. Similarly if
$F_B(\beta_1,a_1)=F_B(\beta_2,a_1)=0$ then $F_A(\alpha_1,a_1)\neq0$ and
$F_A(\alpha_2,a_1)\neq0$.

For $(2)$, notice that $\overline{Scal}$ and $m$ are rational functions
of $\alpha_1$, $\alpha_2$, $\beta_1$, $\beta_2$, $C_{\alpha_1}$, $C_{\alpha_2}$,
$C_{\beta_1}$, $C_{\beta_2}$ and are homogenous and linear with respect to
(inverses of) the latter four variables. Moreover, when $(\Delta,u)$ is rational, the terms $$\frac{C_{\alpha_1}}{C_{\beta_2}}, \frac{C_{\alpha_2}}{C_{\beta_2}}, \frac{C_{\beta_1}}{C_{\beta_2}}$$ are rational functions of $\alpha_1,\alpha_2,\beta_1,\beta_2$ as recalled in Lemma~\ref{expressdesC}. Now, suppose for contradiction that $\alpha_1,\alpha_2,\beta_1,\beta_2,a_1$ are all algebraic and satisfy \eqref{polynEXPequatortho} with $a_1\neq0$. Then, thanks to claim $(1)$, $F_B(\beta_2,a_1)\neq 0$ or $F_A(\alpha_2,a_1)\neq0$ which in turn implies that at least one of the two numbers  $e^{-2a_1(\alpha_2-\alpha_1)}$, $e^{-2a_1(\beta_2-\beta_1)}$ is algebraic. But this cannot be true since $a_1(\alpha_2-\alpha_1)$ and $a_1(\beta_2-\beta_1)$ are non zero algebraic numbers by assumption.

For $(3)$, suppose that $\alpha_2-\alpha_1,\beta_2-\beta_1\in\bQ$. Up to a
dilatation, we can assume that $n=\alpha_2-\alpha_1$ and $k=\beta_2-\beta_1$
are co-prime positive integers. We define the polynomials $P(t)=
F_B(\beta_2,t)^{n} F_A(\alpha_1,t)^{k}$ and $Q(t)=F_A(\alpha_2,t)^{k}
F_B(\beta_1,t)^{n}$ and we will show that either $a_A\neq a_B$ or,
whenever $\alpha_1,\alpha_2,\beta_1,\beta_2$ are algebraic, $a_A=a_B$ is algebraic. Claim (3) then follows since $a_A\neq a_B$ implies that \eqref{polynEXPequatortho} has no solution while $a_A=a_B$ being algebraic contradicts claim (2) unless $a_A=a_B=0$.

Suppose that $\alpha_1,\alpha_2,\beta_1,\beta_2,a_1$ satisfy \eqref{polynEXPequatortho} and that $\alpha_1,\alpha_2,\beta_1,\beta_2$ are algebraic. In particular, $P(t)$ and $Q(t)$ have algebraic coefficients (up to factor $C_{\beta_2}^{-nk}$) as soon as $(\Delta,u)$ is rational.

The relations in \eqref{polynEXPequatortho} imply $P(a_1)-Q(a_1)=0$ which in turn makes $a_1$ the root of the polynomial $P(t)-Q(t)$. Thus, unless this polynomial is trivial, $a_1$ is algebraic since the field of algebraic numbers is algebraically closed. There are cases where $P$ and $Q$ coincide:
When $k=n=1$, one can compute that \begin{equation*}
  \begin{split}
    &2(\alpha_1-\beta_1)(P(t)-Q(t))=-t^4(\alpha_1-\beta_1)\left(\frac{1}{C_{\beta_2}C_{\alpha_1}} - \frac{1}{C_{\beta_1}C_{\alpha_2}}\right) \\ &+\frac{-t^3}{8}\left(\left(\frac{2}{C_{\alpha_1}}+\frac{2}{C_{\alpha_2}}\right)-\left(\frac{2}{C_{\beta_1}}
    +\frac{2}{C_{\beta_2}}\right)\right)\left(\left(\frac{2}{C_{\alpha_1}}-\frac{2}{C_{\alpha_2}}\right)-\left(\frac{2}{C_{\beta_1}}
    -\frac{2}{C_{\beta_2}}\right)\right).
  \end{split}
\end{equation*} In particular, $P\equiv Q$ if and only if $$\frac{C_{\alpha_1}}{C_{\alpha_2}}=\frac{C_{\beta_1}}{C_{\beta_2}}\;\;\mbox{ and }\;\;\frac{2}{C_{\alpha_1}}+\frac{2}{C_{\alpha_2}} = \frac{2}{C_{\beta_1}}+\frac{2}{C_{\beta_2}}.$$ Assuming the first equation and considering the signs of the left and right hand sides of the second equation, we conclude that if $k=n=1$, $P\equiv Q$ if and only if
\begin{equation}\label{kiteCOND}
     C_{\alpha_1}=-C_{\alpha_2} =-C_{\beta_1}=C_{\beta_2}.
\end{equation}
However, if $k=n=1$ and~\eqref{kiteCOND} holds, then $a_A=\frac{-1}{\alpha_1-\beta_1}$ and $a_B=\frac{1}{\alpha_1-\beta_1}$.
Indeed, we just have to verify that $F_A(\alpha_1,t)=F_A(\alpha_2,t)=0$ and $F_B(\beta_1,s)=F_A(\beta_2,s)=0$ whenever $t=\frac{-1}{\alpha_1-\beta_1}$ and $s=\frac{1}{\alpha_1-\beta_1}$. These values are non zero, they satisfy then respectively \eqref{ortho1} and \eqref{ortho2}. By unicity, we have $a_A=t$ and $a_B=s$. In particular, $a_A\neq a_B$ and then there is no non-zero solution to the system~\eqref{polynEXPequatortho}.

It remains to check the case $k\neq n$. We will show that $a_1$ is algebraic even if $P\equiv Q$. If $k\neq1$, let $r$ be a root of the polynomial $F_A(\alpha_2,t)$, then $r\neq0$ and, by considering the sign of the discriminant of $F_A(\alpha_2,t)$, $r$ is real. Moreover, $r$ is a root of $Q(t)$ of multiplicity $k, 2k,k+n$ or $2k+2n$. Since $k$ does not divide $n$, $r$ is also a root of $F_A(\alpha_1,t)$. Hence, $r$ is a non-zero solution of the first line of \eqref{polynEXPequatortho} and thus, by unicity, we have $r=a_A=a_1$ which in turn implies that $a_1$ is algebraic. If $k=1$ and $n\neq 1$, we can repeat the argument with a root of $F_B(\beta_2,t)$. \end{proof}

Recall from~\cite{TGQ}, that each orthotoric quadrilateral $\Delta$ is affinely equivalent to an orthotoric quadrilateral of parameters $\alpha_1=1,\alpha_2=\alpha,\beta_1=0,\beta_2=\beta$. If $\Delta$ is of rational type then $\alpha$ and $\beta$ have the same algebraic type: they are both transcendental or both algebraic and, in that case, they have the same degree. Moreover, $\alpha,\beta\in\bQ$ if and only if the vertices of $\Delta$ are contained in a lattice. Proposition~\ref{propostionNOexistOrtho} follows then from the next lemma together with Lemma~\ref{FactsF_AF_B}(3).

\begin{lemma} A rational labelled orthotoric quadrilateral associated to orthotoric parameters $(\alpha_1,\alpha_2,\beta_1,\beta_2,C_{\alpha_1}, C_{\alpha_2}, C_{\beta_1},C_{\beta_2})$ is monotone if and only if
$$\alpha_2\left( 1+\frac{1}{p} -\frac{1}{l}\right) + \beta_1\left( \frac{1-r}{l} -1-\frac{1}{k}\right) +\beta_2\left(\frac{r}{k}-\frac{1}{p}+\frac{r}{l}\right)=0$$ where $r,p,k,l\in\bQ$ are given as in Lemma~\ref{expressdesC}. In that case, its vertices lie in a lattice.
\end{lemma}
\begin{proof}
  Just replace the occurrences of $C_{\alpha_1}$, $C_{\alpha_2}$, $C_{\beta_1}$, $C_{\beta_2}$ in~\eqref{monotoneCONDOrtho} by using the formulas of Lemma~\ref{expressdesC}. For the last claim, use the particular representative of parameters $(\alpha_1,\alpha_2,\beta_1,\beta_2)=(1,\alpha,0,\beta)$. By using the first condition of Lemma~\ref{expressdesC} we easily show that if $\frac{\alpha}{\beta}\in\bQ$ then $\alpha,\beta\in\bQ$.
\end{proof}

However, there do exist orthotoric generalized K\"ahler--Ricci solitons
with non constant scalar curvature. To see this, first notice that fixing a
rational number $r<0$, for any $\frac{1}{1-r}<\beta<1$ we have
$\alpha= \frac{r\beta}{\beta(r-1) +1}>1$. Thus, the labelled orthotoric
quadrilateral of parameters
$$(1,\alpha,0,\beta, \beta l,-\frac{\beta}{\alpha}k,(\beta-1)p,1)$$ is
rational for any choice of positive rational numbers $k$, $l$ and $p$ as
stated in Lemma~\ref{expressdesC}.
For example, take $r=-1$, $k=1$, $l=2$, $p=3$. By using the uniqueness of a
solution $a_A(\beta)$ of \eqref{ortho1} and $a_B(\beta)$ of \eqref{ortho2}
for any $\beta \in (\frac{1}{2},1)$ and considering the signs of the left hand
sides of these equations for $a_1=0,1,2$ when $\beta$ is $0.6$ and $0.7$
respectively,
we observe that,
$$0<a_A(0.6) < 1, \;1<a_B(0.6) < 2,\; 0<a_B(0.7))<1, \mbox{ and }\;\; 1<a_A(0.7) < 2.$$
Hence, since $a_A(\beta)$ and $a_B(\beta)$ depend continuously on $\beta$ we deduce that there exists $\beta\in(0.6,0.7)$ such that $a_A(\beta)=a_B(\beta)\neq 0$.

\section{K\"ahler--Ricci solitons on weighted projective planes}\label{sectWPP}
\subsection{Calabi K\"ahler--Ricci soliton on weighted projective plane}\label{sectioncalabiTRIANGLE}
When $\alpha_1$ tends to $0$, the labelled Calabi trapezoid of parameters
$$(\alpha_1,\alpha_2,\beta_1,\beta_2,C_{\alpha_1}, C_{\alpha_2},
C_{\beta_1},C_{\beta_2})$$ tends to the {\it Calabi triangle} $\Delta$ of vertices $$(0,0), (\alpha_2, \alpha_2\beta_1), (\alpha_2, \alpha_2\beta_2)$$ and normals \begin{align} \label{labelCalabiTRIANGLE} u_{\alpha_2}=C_{\alpha_2}\begin{pmatrix}
\alpha_2\\
0\end{pmatrix}, \;u_{\beta_1}=C_{\beta_1}\begin{pmatrix}
\beta_1\\
-1\end{pmatrix},\; u_{\beta_2}=C_{\beta_2}\begin{pmatrix}
\beta_2\\
-1\end{pmatrix}.\end{align} Applying Proposition~\ref{defSympPot} for this labelled Calabi triangle $(\Delta, u=\{u_{\alpha_2}, u_{\beta_1},u_{\beta_2}\})$: for any functions $A\in C^{\infty}([0,\alpha_2])$ and $B\in C^{\infty} ([\beta_1,\beta_2])$, the matrix valued function $\bfH_{A,B}$ given in~(\ref{bfHABcalabi}) defines a smooth metric if and only if $A$ and $B$ are positive on the interior of their interval of definition, $\bfH_{A,B}$ is smooth at $(0,0)$ and \begin{align}\label{eq:CondCompactCalabiTRIANGLE}
\begin{split}
A(\alpha_2)=0, &\;\; B(\beta_1)=0,\;\; B(\beta_2)=0 \\
A'(\alpha_2) =\frac{2}{C_{\alpha_2}}, &\;\; B'(\beta_1) =-\frac{2}{C_{\beta_1}}, \;\; B'(\beta_2) =-\frac{2}{C_{\beta_2}}.
\end{split}
\end{align}

Note that Lemma~\ref{lemCalabiEQUIP} also holds in this case. One can compute that $$C= \frac{2}{C_{\alpha_2}} + \frac{\overline{Scal}\alpha_2^2}{2} +m\alpha_2 =0$$ with $m$ given by~\eqref{mCALABI}. We prove then
\begin{lemma}\label{lemCalabiTRIANGLE} Let $(\Delta,u)$ be the labelled Calabi triangle with parameters $(\alpha_2,\beta_1,\beta_2, C_{\alpha_2}, C_{\beta_1},C_{\beta_2})$. If $C_{\beta_1}=-C_{\beta_2}$, the matrix valued function $\bfH_{A,B}$, given by~(\ref{bfHABcalabi}) with the functions $A$ and $B$ given by~\eqref{solAcalabi} and \eqref{fctBcalabi} where $C=0$ and $\alpha_1=0$ and where $a_1$ is the unique solution of equation~(\ref{condEXISTcalabi}), is a solution of equation~\eqref{KRStraceequation} of $(\Delta,u)$.\end{lemma}

\begin{proof}
 We only have to show that $\bfH_{A,B}$ is smooth at $(0,0)$. If $a_1=0$ then $$A(x)=-\left(\frac{\overline{Scal}}{6} x^3 +\frac{m}{2}x^2\right).$$ We verify easily then that $\bfH_{A,B}$ is smooth.
  Suppose now that $a_1\neq 0$, then \begin{equation}
    \begin{split}
      A(x)= \frac{\overline{Scal}}{2} \left(\frac{x^2}{2a_1}+\frac{x}{2a_1^2}+\frac{1}{4a_1^3} -\frac{e^{2a_1x}}{4a_1^3}\right)+ m \left(\frac{x}{2a_1}+\frac{1}{4a_1^2} -\frac{e^{2a_1x}}{4a_1^2}\right)\\
    \end{split}
  \end{equation}Again, we verify easily that in that case $\bfH_{A,B}$ is smooth at $(0,0)$.
\end{proof}

\begin{proposition}
  A weighted projective plane having two equal weights $\bcp_{(l,k,k)}^2$ admits a compatible toric K\"ahler--Ricci soliton $(g,a)$ with a Hamiltonian $2$--form of order $1$ in each K\"ahler class. In particular, $g$ is explicitly given in terms of $2$ functions $A$, $B$ themselves explicitly determined by the K\"ahler class and the weights.\end{proposition}
\begin{proof}
  Each K\"ahler class on a given weighted projective plane $\bcp_{(k_1,k_2,k_3)}^2$ admits $T$--invariant symplectic form and thus corresponds to a labelled triangle $(\Delta,u)$, $\Delta \subset \kt^*$. The choice of the symplectic form and the torus does not matter up to an invertible affine linear map. The weights are the entries of the primitive vector $(k_1,k_2,k_3)\in\bZ^3$ such that $k_1u_1 +k_2u_2 +k_3u_3=0.$ There is a unique choice of such vector so the entries are all positive. Put $k_3=k_2=k$ and let $l=k_1$. There is an affine identification between $\kt$ and $\bR^2$ so that $\Delta$ is identified to the convex hull of $\{(0,0), (1,0), (1,1)\}$. That is, $(\Delta,u)$ is identified to the Calabi triangle with parameters $\alpha_2=1$, $\beta_1=0$, $\beta_2=1$ and normals $u_{1}=u_{\alpha_2}$, $u_{2}=u_{\beta_1}$ and $u_{3}=C_{\beta_2}$ with $-lC_{\alpha_2}=kC_{\beta_2}$ and $C_{\beta_2}=-C_{\beta_1}$. The exact value is given by the K\"ahler class considering $\overline{Scal}= -\frac{4}{C_{\alpha_2}} + \frac{8}{C_{\beta_2}}.$ We are exactly in the hypothesis of Lemma~\ref{lemCalabiTRIANGLE} which gives an explicit solution of~\eqref{KRStraceequation}. In our case, this solution is a K\"ahler--Ricci soliton since any K\"ahler class on a weighted projective plane is monotone. \end{proof}

\subsection{Orthotoric K\"ahler--Ricci soliton on weighted projective plane}\label{wpp}

As a border-line case of the orthotoric quadrilateral setting, the
orthotoric simplex case corresponds to setting $\alpha_{1} = \beta_{2}$
and $C_{\alpha_{1}} = C_{\beta_{2}}$ in Section \ref{sectionOrtho}.
For background on this see \cite{H2FII} and in particular Proposition 7.
Beware that the notation in \cite{H2FII} and in particular the use of alphas and betas
is slightly different.

As can be seen from e.g. Theorem 3 in \cite{H2FII}, the rational Delzant polytope
case corresponds to setting
\begin{equation}
\begin{array}{ccl}
C_{\beta_{1}} & = &\frac{-2 n_{0}}{c(\beta_{1}-\beta_{2})(\beta_{1}-
\alpha_{2})}\\
\\
C_{\beta_{2}} & = &\frac{-2 n_{1}}{c(\beta_{2}-\beta_{1})(\beta_{2}-
\alpha_{2})}\\
\\
C_{\alpha_{2}} & = &\frac{-2 n_{2}}{c(\alpha_{2}-\beta_{1})(\alpha_{2}-
\beta_{2})},
\end{array}
\end{equation}
where $c>0$ and $n_{i} \in \Z^{+}$.
The simplex is then associated to an orbifold equivariantly
biholomorphic to a toric orbifold quotient of the
weighted projective space ${\mathbb C}P^{2}_{(k_{1}, k_{2}, k_{3})}$,
where $k_{1},k_{2},k_{3} \in \Z^{+}$, $\gcd\{k_{1},k_{2},k_{3} \} =1$
(without loss), and
$n_{i} = \prod_{j \neq i} k_{j}$.
In fact,
\[
\begin{array}{ccl}
C_{\beta_{1}}= \frac{k_{2}}{k_{1}}\frac{(\alpha_{2}-
\beta_{2})}{(\beta_{1}-\alpha_{2})} C_{\beta_{2}}\\
\\
C_{\alpha_{2}}=\frac{k_{2}}{k_{3}}\frac{(
\beta_{2}- \beta_{1})}{(\beta_{1}-\alpha_{2})} C_{\beta_{2}},
\end{array}
\]
or
\[
\begin{array}{rcc}
k_{2} C_{\beta_{2}} \beta_{2} + k_{3} C_{\alpha_{2}} \alpha_{2} +
k_{1}C_{\beta_{1}}\beta_{1} & = & 0\\
\\
k_{2} C_{\beta_{2}}  + k_{3} C_{\alpha_{2}}  +
k_{1}C_{\beta_{1}} & = & 0,
\end{array}
\]
so the normals generate a lattice according to Lemma 6.7 in
\cite{TGQ}. This is exactly the condition for the simplex being a
rational labelled polytope.

We may, without loss, assume that $\beta_{1}=-1$,
$\beta_{2}=\alpha_{1}=\beta$, and $\alpha_{2}=1$. We then have
\begin{equation}\label{orthosimplexrat}
\begin{array}{ccl}
C_{-1}= \frac{k_{2}}{k_{1}}\frac{(
\beta-1)}{2} C_{\beta}\\
\\
C_{1}=\frac{-k_{2}}{k_{3}}\frac{(
\beta+1)}{2} C_{\beta},
\end{array}
\end{equation}

\begin{lemma}\label{orthosimplex3}
For any values of $t=\frac{k_{2}}{k_{1}} \in (1,+\infty)$ and
$s=\frac{k_{2}}{k_{3}} \in (0,1)$, there exists a pair $(\beta,a_{1})
\in (-1,1) \times {\mathbb R}$ such that \eqref{ortho1} and
\eqref{ortho2} are both satisfied.
\end{lemma}

\begin{proof}
Consider $t=\frac{k_{2}}{k_{1}} \in (1,+\infty)$ and
$s=\frac{k_{2}}{k_{3}} \in (0,1)$ fixed. We then insert
\begin{equation}\label{orthosimplexrat2}
\begin{array}{ccl}
C_{-1}=\frac{t(
\beta-1)}{2} C_{\beta}\\
\\
C_{1}=\frac{-s(
\beta+1)}{2} C_{\beta},
\end{array}
\end{equation}
as well as our other assumptions into the equations for
$f_{A}(x)$ and $f_{B}(y)$ from
\eqref{ortho1} and \eqref{ortho2} and arrive at
\[
\begin{array}{ccl}
f_{A}(x) & = & \frac{2f_{\beta}(x)}{s\,t\,(1-\beta^{2})C_{\beta}}\\
\\
f_{B}(y) & = & \frac{-2f_{\beta}(y)}{s\,t\,(1-\beta^{2})C_{\beta}},
\end{array}
\]
where
\[ f_{\beta}(z) = -(s+t+st)z^{2} +(s(1+\beta) +t(\beta-1)) z +st
+\beta(t-s).
\]
It is easy to check that for any $-1 \leq \beta <1$, $f_{\beta}(1)<0$, $f_{\beta}(\beta)=st(1-\beta^2)\geq 0$ and $f_{-1}'(-1)>0$. Thus $f_{\beta}(z)$ has precisely one root in the interval $(\beta,1)$. Therefore, by the usual argument, for any
$-1 \leq \beta <1$, the equation
\begin{equation}\label{orthosimplex1}
\int_{\beta}^{1}f_{\beta}(x) e^{-2a_{1}x}\, dx =0
\end{equation}
has a unique solution $a_{1}=a_{A}(\beta)$. For $-1<\beta<1$
\eqref{orthosimplex1} is equivalent to \eqref{ortho1} as applied to
the present case. Likewise, for any $-1 < \beta \leq 1$, $f_{\beta}(z)$ is
negative
at $z= -1$, positive or zero at $z=\beta$ and positive to the immediate
left
of $z=\beta$. Thus $f_{\beta}(z)$ has precisely one root in the
interval $(-1,\beta)$. Therefore, for any
$-1 < \beta \leq 1$, the equation
\begin{equation}\label{orthosimplex2}
\int_{-1}^{\beta}f_{\beta}(y) e^{-2a_{1}y}\, dy =0
\end{equation}
has a unique solution $a_{1}=a_{B}(\beta)$. For $-1<\beta<1$
\eqref{orthosimplex2} is equivalent to \eqref{ortho2} as applied to
the present case.

Now we have the graphs of two continuous functions, namely $a_{A}(\beta)$,  $-1\leq\beta < 1$
and $a_{B}(\beta)$,  $-1< \beta \leq 1$, implicitly given by
\eqref{orthosimplex1} and\eqref{orthosimplex2}. We are done with the
proof if we can show that these two graphs must intersect for some
$-1 < \beta <1$.

Consider the left hand side of \eqref{orthosimplex1} as a smooth function,
$g_{A}(\beta,a_{1})$, of the two variables
$-1\leq \beta \leq 1$ and $a_{1}\in {\mathbb R}$. Likewise let
$g_{B}(\beta,a_{1})$ denote the left hand side of \eqref{orthosimplex2}.
We now observe that
\[
g_{A}(0,0) = \frac{(4st+s-5t)}{6}< 0\]
and
\[
g_{B}(0,0)= \frac{(t+4st-5s)}{6} > 0.\]
As the discussion above implies, for any $-1\leq \beta<1$, the map $a_1\mapsto g_A(\beta,a_1)$ vanishes only once and $ g_{A}(0,a_{1})$ is positive for $a_1$ big enough since
\[
g_{A}(0,a_{1}) \sim \frac{s\,t}{2a_1}=0, \qquad \qquad a_{1}\rightarrow +\infty.\]
Likewise, for any $-1< \beta\leq1$, the map $a_1\mapsto g_B(\beta,a_1)$ vanishes only once and $ g_{B}(0,a_{1})$ is negative for $a_1$ big enough since
\[
g_{B}(0,a_{1}) \sim
\frac{-
e^{2a_{1}}s}{a_{1}}=-\infty, \qquad \qquad a_{1}\rightarrow +\infty
\]
Therefore, we clearly get that
$a_{A}(0) > 0$ and $a_{B}(0) >0$ and, more importantly,
\[ \forall a_{1}>a_{A}(0), \quad g_{A}(0,a_{1})>0\]
and
\[ \forall a_{1}>a_{B}(0), \quad g_{B}(0,a_{1})<0.\]

A calculation shows that for any $a_{1}$
\[
g_{A}(1,a_{1})=\frac{\partial g_{A}}{\partial \beta}(1,a_{1})=0
\]
while
\[
\frac{\partial^{2} g_{A}}{\partial
\beta^{2}}(1,a_{1})=2e^{-2a_{1}}(s-1)t <0.
\]
Likewise
\[
g_{B}(-1,a_{1})=\frac{\partial g_{B}}{\partial \beta}(-1,a_{1})=0
\]
while
\[
\frac{\partial^{2} g_{B}}{\partial \beta^{2}}(-1,a_{1})=2e^{2a_{1}}s(t-1) >0.
\]
From these observations we conclude that
\[
\forall a_{1} > a_{A}(0), \quad \exists \beta \in (0,1),\quad
\mbox{s.t.} \quad a_{A}(\beta)=a_{1}\]
and
\[
\forall a_{1} > a_{B}(0), \quad \exists \beta \in (-1,0),\quad
\mbox{s.t.}\quad  a_{B}(\beta)=a_{1}.
\]
Thus we have that the graph of $a_{A}(\beta)$ is
bounded for $-1 \leq \beta \leq 0$ (by continuity) and unbounded
above for
$0 \leq \beta < 1$ while on the other hand the graph of  $a_{B}(\beta)$
is bounded for $0 \leq \beta \leq 1$ and unbounded above
for $-1 < \beta \leq 0$. It is now clear that these graphs must
intersect for some $-1 < \beta < 1$.
\end{proof}

From Lemma \ref{orthosimplex3}, the discussion in
Section \ref{sectionOrtho}, and the fact that any K\"ahler class on
the weighted projective plane is monotone, we may conclude with the
following proposition.

\begin{proposition}
A weighted projective plane  ${\Bbb C}P^{2}_{(k_{1}, k_{2}, k_{3})}$
with distinct weights admits a compatible toric K\"ahler--Ricci soliton $(g,a)$ with a Hamiltonian
$2$--form of order $2$ in each K\"ahler class.
\end{proposition}

\section{Toric Sasaki--Ricci solitons}
We briefly recall the main points of toric Sasakian geometry and refer to \cite[Chapter 8]{BG} for more details.

Given a co-oriented compact contact manifold $(N^{2n+1}, \bfD)$, we denote by $(\bfD_+^o,\hat{\omega})$ its symplectisation, where we see $\bfD_+^o\subset T^*N\backslash \{0-\mbox{section}\}$ and $\hat{\omega}$ is the restriction of the canonical symplectic form on $T^*N$. We take the convention that $\mL_\tau\hat{\omega}=2\hat{\omega}$ where $\tau$ denotes the Liouville vector field. Recall that a Sasaki metric $g$ on $(N, \bfD)$ corresponds to a K\"ahler cone metric $\hat{g}$ on $(\bfD^o_+,\hat{\omega})$, that is, $\hat{g}$ is homogeneous of degree $2$ with respect to $\tau$ (which is real holomorphic) and coincides with $g$ on $N$, seen as the level set $\hat{g}(\tau,\tau)=1$. In the {\it toric} case, $(\bfD_+^o,\hat{\omega})$ is a toric symplectic cone meaning that, on top of being toric, the action of the torus $\hat{T}$ commutes with $\tau$ so the action is defined on $N$ where it preserves $\bfD$ and its co-orientation. We denote $\hat{\kt}:=\mbox{Lie}\,\hat{T}$ and $\hat{\mu}: \bfD_+^o \ra \hat{\mathfrak{t}}^*$ the unique moment map of $(\bfD_+^o,\hat{\omega}, \hat{T})$ which is homogeneous of degree $2$ with respect to $\tau$. For any $\hat{T}$--invariant contact form $\eta: N\ra \bfD_+^o$ and $a\in\hat{\kt}$, $\eta(X_a)=\langle\hat{\mu}_{\eta},a\rangle$. The {\it moment cone} is $\mC=\im \hat{\mu}\cup \{0\}$.
It is known that the Reeb vector field of a toric Sasaki manifold $(N, \bfD,g,\hat{T})$ is induced by an element $b\in\hat{\kt}$. Such a vector $b$ must lie in $\mC_+^*$, the interior of the dual cone of $\mC$ (the set of strictly positive linear maps on $\im \hat{\mu} =\mC\backslash \{0\}$). In particular, $\mC$ is a {\it strictly convex} polyhedral cone, that is, $\mC_+^*$ is not empty. From~\cite{contactNOTE,L:contactToric} we know that toric contact manifolds of Reeb type of dimension at least $5$ are in correspondence with strictly convex polyhedral cones $\mC$ which are \textit{good} with respect to a lattice $\Lambda$. This means that every set of primitive vectors normal to a face of $\mC$ can be completed to a basis of $\Lambda$.

Given a strictly convex polyhedral cone $\mC$, which is good with respect to a lattice $\Lambda$, one can associate to any
$b\in\mC_+^*$ the {\it characteristic labelled polytope} $(\Delta_b,u_b)$ where
$$\Delta_b=\mC\cap \left\{ y \,\left|\, \langle b,y\rangle =\frac{1}{2}\right.\right\}$$ is a compact simple polytope and $u_b=\{[u_{1}]_b,\dots, [u_{d}]_b\}$ is the set of equivalence classes in $ \hat{\kt}/\bR b$ of the primitive vectors of $\Lambda$ which are inward normal to the facets of $\mC$. Here, $ \hat{\kt}/\bR b$ is identified with the dual vector space of the annihilator of $b$ in $\hat{\kt}^*$ which, in turn, is identified with the hyperplane $\{ y\,|\, \langle b,y\rangle =\frac{1}{2}\}$.

A $p$--form $\psi$ (or a tensor) satisfying $\psi(X_b)=0$ and $\mL_{X_b}\psi =0$ is called {\it basic}. For instance, denoting $\eta_b$ the contact form of the Reeb vector field $X_b$, $d\eta_b$ is basic. The exterior differential preserves this property and one defines the {\it basic cohomology}, whose classes are denoted $[\psi]_B$. This is the relevant cohomology to study the {\it transversal K\"ahler structure} of a Sasakian manifold, referring to the triple formed by the $2$--form $d\eta_b$ playing the role of the symplectic form, the CR-structure on $\bfD$ induced by the inclusion in a K\"ahler cone, the metric $\check{g}$ on $\bfD$ given by $g= \eta_b\otimes \eta_b +\check{g}$. This is well explained in~\cite{FuOnWa} where the transversal geometry is described through local submersions on $\bC^n$ patched together by K\"ahler isometries. This gives local identifications between $N$ and $\bC^n\times \bR$ so the Sasakian structure gives rise to a K\"ahler structure on (subsets of) the first factor and one can compute curvatures $R^\top, \mbox{Ric}^\top, \mbox{Scal}^\top$... In particular, the transverse Ricci form $\rho^\top$ is a (well-defined) basic closed form defining a class $c_1^B$. The condition $c_1^B=\lambda [d\eta_b]_B$, for $\lambda>0$, implies that $c_1^B>0$ and $c_1(\bfD)=0$. In the toric case, this condition is equivalent to the fact that the characteristic labelled polytopes are monotone.

A Sasaki--Ricci soliton is a Sasakian structure satisfying
\begin{equation}\label{SRsoliton}
  \rho^\top -\lambda d\eta_b = \mL_Yd\eta_b
\end{equation} for some $\lambda>0$ \footnote{Usually, one chooses $\lambda =m+2$ to get Sasaki--Einstein metrics (and not only $\eta$--Einstein ones). Note that $m+2$ can be obtained from any other choice by a transversal homothety.} and a {\it Hamiltonian holomorphic vector field} $Y$, that is, the projection from the K\"ahler cone  of a (usual) Hamiltonian holomorphic vector field generated by a basic Hamiltonian function.   In \cite{FuOnWa}, Futaki--Ono--Wang proved the existence of a compatible (toric) Sasaki--Ricci soliton on any compact toric contact manifold, with a fixed Reeb vector field, such that $c_1^B>0$ and $c_1(\bfD)=0$.

Fixing the Reeb vector field $X_b$ and the contact structure $\mD$, the set of symplectic potentials $\mS(\Delta_b,u_b)$ parametrizes the transversal K\"ahler structures so that each of them, together with $X_b$, determines a toric Sasakian structure on $(N,\mD,\hat{T})$. The correspondence is made explicit by the Boothy--Wang construction which associates to any $u\in \mS(\Delta_b,u_b)$ a function on $\mathring{\mC}$ with the boundary condition required to give a smooth K\"ahler cone metric on the symplectisation, see~\cite{abreuSasakAAcoord,ENU,MSYvolume}.

In the toric case, a KR-soliton on the characteristic labelled polytope can be lifted as a SR-soliton as they satisfy the same equation \cite[Section 5]{FuOnWa}. Moreover, the SR-vector $Y$ is the projection of $JX_a$ where $a$ is the KR-soliton vector of the characteristic labelled polytope $(\Delta_b,u_b)$ as proved in \cite[Lemma 7.5]{FuOnWa}.  Of course such a vector is well defined only up to addition by a multiple of $b$ but since  $-JX_b$ is the Liouville vector field this multiple does not change the projection $Y$.

Note that the proofs of Theorems~\ref{theoWPP} and~\ref{theoPrincipal} only use the data of labelled polytopes without any rational assumption.  Thus, their conclusions apply also to characteristic labelled polytopes and we get the following result and Theorem~\ref{theoS5}.

\begin{proposition}
  A $5$--dimensional toric Sasaki--Ricci soliton $(S,X_b,\hat{g},\hat{T})$ with Hamiltonian vector field corresponding to $a\in \Lie \hat{T}$ admits a transversal Hamiltonian $2$--form if and only if either the moment cone has $3$ facets or 4 facets and $a$ is equipoised on the characteristic polytope.
\end{proposition}

We now prove that there exists a continuous family of toric Sasaki--Ricci solitons on $S^2 \times S^3$ having a transversal Calabi type metric, that is, admitting a transversal hamiltonian $2$-form of order 1.  We use the framework developed in \cite{ENU}.


Given a labelled polytope $(\Delta,u)$ with defining functions $L_1,\dots, L_d \in \mbox{Aff}(\kt^*,\bR)$, we denote by $C(\Delta)$ the cone over $\Delta$
 $$C(\Delta)=\{y\in \mbox{Aff}(\kt^*,\bR)^*\,|\, \langle y, L_k \rangle \geq 0 \}$$
and $(\Delta,u)$ is transversal to a good cone if and only if $L_1,\dots, L_d$ are primitive elements of a lattice, say $\Lambda$, of $\mbox{Aff}(\kt^*,\bR)$ for which $C(\Delta)$ is good. The dual cone of $C(\Delta)$ is $C(\Delta)^* =\{ b\in  \mbox{Aff}(\kt^*,\bR)\,|\,b(x) >0,  \; \forall x\in \Delta\}.$ The characteristic polytope to $(C(\Delta),\Lambda)$ at $b\in C(\Delta)^*$ lies in $\kt_b^*=\{y\in\mbox{Aff}(\kt^*,\bR)^*\,|\, \langle y, b \rangle =\frac{1}{2}\}$ which, up to a translation,  is the dual of $\kt_b= \quot{\mbox{Aff}(\kt^*,\bR)}{\bR b}.$  There is a diffeomorphism $\Psi_b : \Delta \ra \Delta_b$ given by $$\Psi_b(\mu) = \frac{\ev_{\mu}}{2b(\mu)}$$ where $\ev$ is the evaluation map. In particular, if $(\Delta,u)$ is monotone with preferred point $p\in\mathring{\Delta}$ then $(\Delta_b, u_b)$ is monotone of preferred point $\Psi_b(p)$. In that case, for each $b$ one can define the KR-vector $[a]_b \in \kt_b$ of $(\Delta_b, u_b)$ as the unique critical point (in $\kt_b$) of the $\bR b$--invariant function on $F_b : \mbox{Aff}(\kt^*,\bR)\ra \bR$ defined as $$F_b(a)=\int_{\Delta_b} \exp^{\langle a, y-\Psi_b(p)\rangle} \vol_b$$ where $\vol_b$ is any affine volume form on $\kt^*_b$ and $y$ is the variable of integration.

When $\Delta$ is a quadrilateral, for $v\in \mbox{Aff}(\kt^*,\bR)$, the fact that $v$ is equipoised on $\Delta_b$ does not depend on the representative of $v$ in $[v]_b$ nor on a translation of $\Delta_b$. Moreover, $v$ is equipoised on $\Delta_b$ if and only if $\Psi_b^*v=\frac{v}{2b}$ is equipoised on $\Delta$.

Consider a square $\Delta=[-1,1]\times [-1,1] \subset \bR^2$, labelled so that it is Delzant with respect to $\bZ^2 \subset \bR^2$. Identifying $\mbox{Aff}(\bR^2,\bR)$ with $\bR^3$, we get a good cone corresponding to the ``simplest" contact toric structure on $S^2\times S^3$.
Denote the vertices of $\Delta$, $p_1=(-1,-1)$, $p_2=(1,-1)$, $p_3=(1,1)$ and $p_4=(-1,1)$. Observe that if $b\in \mbox{Aff}(\bR^2,\bR)$ satisfies \begin{equation}\label{CalabiKSrel}
  b(p_1)= b(p_2)\neq b(p_3)= b(p_4),
\end{equation} then $\Delta_b$ is a trapezoid (but not a parallelogram) with a reflection symmetry preserving the labelling. By uniqueness, the KR-vector $[a]_b$ is also invariant by this involution and, thus, is equipoised on $\Delta_b$ since $$\sum_{i=1}^4(-1)^i\frac{a(p_i)}{2b(p_i)}=0.$$

Writing $b(\mu)= b_0+ b_1\mu_1+ b_2\mu_2$, condition \eqref{CalabiKSrel} corresponds to $b_1=0$ and we get a $2$--real parameters continuous family of explicit Sasaki--Ricci solitons on $S^2\times S^3$ having a transversal Calabi type metric. Elements of this family are regular, quasi-regular or irregular depending on the ratio $b_2/b_0$.

\begin{rem} We can not apply the argument above to another toric contact structure on $S^2\times S^3$. Indeed, up to an affine transform, all strictly convex cones with $4$ facets in $\bR^3$ are the same, without loss of generality, we can only consider other lattices on the same cone. However, condition \eqref{CalabiKSrel} together with monotonicity, see ~\eqref{monotoneCONDprod}, imply that the labelling is the one above.
\end{rem}
%
%

\bibliographystyle{abbrv}

\end{document}